\documentclass[oneside,a4paper,11pt]{article}
\usepackage{mathtools}
\usepackage{amsmath}
\allowdisplaybreaks[4]
\usepackage{csquotes}
\usepackage{amsthm}
\usepackage{amssymb}
\usepackage{mathrsfs}
\usepackage{color}
\usepackage{bm}
\usepackage{comment}
\usepackage{fancyhdr}
\usepackage{geometry}
\usepackage{hyperref}
\usepackage{enumerate}
\usepackage{graphicx}
\usepackage{subfig}
\usepackage{float}

\newtheorem{conjecture}{Conjecture}[section]
\newcommand{\pp}[2]{\frac{\partial#1}{\partial#2}}
\theoremstyle{definition}
\newtheorem{thm}{Theorem}[section]
\newtheorem{cor}[thm]{Corollary}
\newtheorem{lem}[thm]{Lemma}
\newtheorem{prop}[thm]{Proposition}
\newtheorem{defn}[thm]{Definition}
\newtheorem{rem}[thm]{Remark}

\numberwithin{equation}{section}
\newcommand{\ddt}[1]{\frac{\mathrm{d}#1}{\mathrm{d}t}}

\title{Combinatorial Ricci flows on infinite disk triangulations}
\author{Huabin Ge, Bobo Hua, Puchun Zhou
}

\newcommand{\R}{\mathbb{R}}
\newcommand{\Z}{\mathbb{Z}}
\newcommand{\EE}{\mathcal{E}}

\date{}
\usepackage{fancyhdr}
\providecommand{\classification}[1]
{
	\small	
	\textbf{Mathematics Subject Classification (2020):} #1
}

\begin{document}
	\maketitle
	\begin{abstract}
	In this paper, we introduce combinatorial Ricci flows (CRFs in short) in Euclidean and hyperbolic background geometries on infinite triangulations of the open disk, which are discrete analogs of Ricci flows on simply connected open surfaces. We establish well-posedness results, the existence and the uniqueness, of CRFs in both Euclidean and hyperbolic background geometries.  Moreover, we prove convergence results of CRFs, which indicate a uniformization theorem for CRFs on infinite disk triangulations. As an application, we prove an existence result of circle-packing metrics with infinite prescribed cone angles in hyperbolic background geometry.  
  To our knowledge, these are the first results of CRFs on infinite triangulations.
   \\
\classification{52C26, 51M10, 57M50}	
		
	\end{abstract}

\section{Introduction} \label{sec:1}
The circle packing is a discrete analog of the conformal structure on a Riemann surface, which Thurston rediscovered in \cite{thurston1980geometry} for studying three-dimensional hyperbolic manifolds. Rodin and Sullivan proved that hexagonal circle packings can be used to approximate the Riemann mapping \cite{Rodin_Sullivan}, confirming a conjecture of Thurston. The first theorem of the existence of circle packings on surfaces was proven by Koebe \cite{koebe1936origin}, which is now called the Koebe-Andreev-Thurston theorem: every finite simple planar graph $G$ admits a circle packing on the plane $\mathbb{R}^2$ whose contact graph is $G$. Later, Colin de Verdi\`ere \cite{MR1106755} gave an alternative proof of the existence and uniqueness of circle packings with a variational method. 

Inspired by Hamilton's Ricci flow and the convergence of Ricci flows on compact surfaces \cite{ham,Chow1}, Chow and Luo \cite{2003Combinatorial}  introduced combinatorial Ricci flows (CRFs in short) for finding circle-packing metrics with zero discrete Gaussian curvatures on triangulations $\mathcal{T}=(V,E,F)$ of compact surfaces,
\begin{align}
    &\ddt{r_i}=-K_ir_i,~~\forall i\in V,\quad \mathrm{in\ Euclidean\ background\ geometry},\label{unnormalized}\\
    &\ddt{r_i}=-K_i\sinh {r_i},~~\forall i\in V, \quad \mathrm{in\ hyperbolic\ background\ geometry,}\label{hyperbolic_flow}
\end{align} where $r_i(t)$ ($K_i$ resp.) is the radius of the circle centered (discrete Gaussian curvature resp.) at the vertex $i$ for $t\geq 0.$ These can be regarded as some discrete variants of fully nonlinear parabolic partial differential equations.
They proved that the CRF (with some possible modification) converges if the intersection angle $\Phi\in [0,\frac{\pi}{2}]^E,$ i.e. $\Phi:E\to [0,\frac{\pi}{2}],$ satisfies certain combinatorial conditions. A circle-packing metric $r$ with the intersection angle $\Phi$ on a triangulation $\mathcal{T}$ will be defined in Section \ref{sec:2}.
Later, other combinatorial curvature flows had been introduced for studying discrete conformal structures, such as the combinatorial Calabi flow \cite{ge_phd} and the combinatorial Yamabe flow introduced by Luo \cite{luo2004combinatorial}. In addition, combinatorial curvature flows are also introduced for studying 3-dimensional sphere packings and hyperbolic polyhedra; see e.g. \cite{MR3269185,MR4466650}.

%\red{Add more flows, 2dimensional case? 3 dimensional cases, Luo, Glistein}

It is an important topic to study circle packings on infinite triangulations of noncompact surfaces. The first result of the rigidity was obtained by Rodin and Sullivan in \cite{Rodin_Sullivan}: all the univalent circle packings of the hexagonal triangulation are regular hexagonal packings, as shown in Figure \ref{regular}, where all circles have the same radii. After that, the results on the classification of circle packings on infinite triangulations of the plane $\mathbb{C}$ were proved by He-Schramm \cite{He_schramm} and He \cite{HE}. All these results are based on approaches initiated from elliptic partial differential equations or conformal geometry. A natural question is how to use parabolic methods to study circle packings on infinite triangulations. 

We recall some results on Ricci flows on noncompact manifolds. The well-posedness of the Ricci flow on a noncompact manifold is the key problem. On the one hand, Shi  \cite{Shi_noncompact}  first proved the short-time existence of the Ricci flow on a noncompact Riemannian manifold if the Riemann curvature tensor of the initial metric is uniformly bounded. See e.g. \cite{MR2832165,MR3091259,MR3728651,MR3429162,MR3480020,MR3958792} for more results on the existence of Ricci flows. On the other hand, the uniqueness of a parabolic equation on a noncompact manifolds is a tricky problem. Even for the linear heat equation, the uniqueness of the Cauchy problem fails by the well known example of Tychnoff, and holds under additional assumptions \cite{MR834612,MR860324}. On a noncompact manifold, Chen and Zhu \cite{MR2260930} proved the uniqueness for the solution of the Ricci flow with uniformly bounded curvature.  Chen \cite{MR2520796} proved the uniqueness of the Ricci flow on a three-dimensional noncompact manifold if the initial metric has nonnegative and bounded sectional curvature.  See \cite{MR4015429,MR4494617} for other related results. Note that the existence and uniqueness of Ricci flows on surfaces were extensively studied by Giesen-Topping \cite{MR2832165} and Topping \cite{MR3352241}, which even include results for incomplete Riemann surfaces with unbounded curvature; see also \cite{ToppingYin24}.

In this paper, we first introduce the CRF on an infinite disk triangulation in either Euclidean or hyperbolic background geometry. Here, a \textbf{disk triangulation} refers to a triangulation that is simply connected with only one end. We prove the well-posedness result and the convergence result in both Euclidean and hyperbolic background geometry, which can be regarded as discrete counterparts of the results by Shi \cite{Shi_noncompact} and Chen-Zhu \cite{MR2260930}.
\begin{thm}[Well-posedness of the CRF]\label{wellposed-Euc}

Let $\mathcal{T}=(V,E,F)$ be an infinite disk triangulation with intersection angle $\Phi\in [0,\frac{\pi}{2}]^E.$ For any initial data $r(0),$ there exists a solution $r(t), t\in[0,\infty),$ to the flow \eqref{unnormalized} or \eqref{hyperbolic_flow} . Moreover,  the solution to the flow with uniformly bounded discrete Gaussian curvatures on $V\times[0,\infty)$ is unique.
\end{thm}
\begin{comment}
\textcolor{red}{
For the CRF in Euclidean background geometry, we obtain a better result, which only assumes the boundedness of discrete Gaussian curvatures of the initial circle-packing metric.
\begin{thm}
    Let $\mathcal{T}=(V,E,F)$ be an infinite disk triangulation with intersection angle $\Phi\in [0,\frac{\pi}{2}].$ If the discrete Gaussian curvature of the initial circle-packing metric is bounded, then the solution to the flow \eqref{unnormalized} is unique.
\end{thm}}
\end{comment}
\begin{rem}
    The assumption that $\Phi\in[0,\frac{\pi}{2}]^E$ in the theorem can be extended to $\Phi\in[0,\pi)^E$ with some additional conditions as in the works \cite{zhou2023generalizing,MR4334399}.
\end{rem}
In our setting, due to the discrete nature, the discrete Gaussian curvature is uniformly bounded if the combinatorial degree is uniformly bounded; see Proposition~\ref{prop:degree}. This yields the uniqueness of the CRF without assuming any growth condition of the solution in the spirit of Chen's uniqueness result \cite{MR2520796}, while we don't assume the positivity of the curvatures for the initial data.
\begin{thm}\label{thm:bdd}
    Let $\mathcal{T}=(V,E,F)$ be an infinite disk triangulation with uniformly bounded combinatorial degree and with intersection angle $\Phi\in [0,\frac{\pi}{2}]^E.$ For any initial data $r(0),$ there exists a unique solution $r(t), t\in[0,\infty),$ to the flow \eqref{unnormalized} or \eqref{hyperbolic_flow}.
\end{thm}
For applications, the aim is to study the convergence of the Ricci flow and to find the limit metric with zero curvature. This is in general a hard problem for noncompact manifolds. In the continuous case, several convergence results were proved for Ricci flows on $\mathbb{R}^2$; see e.g. \cite{wu1993ricci, MR1371208, MR2538937, MR3049633}. For other open surfaces, the convergence of Ricci flows was established when assuming that the surface with the initial metric is nonparabolic; see \cite{MR2520032}. For further  results, we refer to \cite{MR2545867, MR3156988}. In our setting, we prove the convergence of the CRF in the hyperbolic background geometry if the initial metric has non-positive curvature.

\begin{thm}[Convergence of the hyperbolic CRF]\label{converge_hyp}

    Let $\mathcal{T}=(V,E,F)$ be an infinite disk triangulation with intersection angle $\Phi\in [0,\frac{\pi}{2}]^E.$
    Let $r(0)$ be a circle-packing metric in hyperbolic background geometry with non-positive discrete Gaussian curvatures. Then there exists a solution $r(t)$ to the flow \eqref{hyperbolic_flow} with the initial value $r(0)$, which converges to a circle-packing metric with zero discrete Gaussian curvatures.

\end{thm}
\begin{rem}
 The existence of a zero-curvature metric under the assumption of non-positive curvature was previously established for finite triangulations of compact surfaces; see e.g. \cite{2003Combinatorial,MR4761889}, which is open for infinite triangulations. For an infinite disk triangulation, we prove the result using the parabolic method, specifically the CRF, while a proof via elliptic methods appears to be unknown. 
 \end{rem}
 
 %The existence of a zero-curvature metric under the assumption of non-positive curvature was previously established for finite triangulations of compact surfaces; see e.g. \cite{2003Combinatorial,MR4761889}. For infinite triangulations, we prove our result using the parabolic method, specifically the CRF, while a proof via elliptic methods appears to be unknown. 
 As a direct application, we have the following corollary, which links the geometry of the circle-packing metric to the random walk on the $1-$skeleton of the corresponding triangulation.

\begin{cor} \label{negative}
    
    Let $\mathcal{T}=(V,E,F)$ be an infinite disk triangulation with intersection angle $\Phi\in [0,\frac{\pi}{2}]^E.$ If $\mathcal{T}$ attains a circle-packing metric with non-positive discrete Gaussian curvatures in hyperbolic background geometry, then $\mathcal{T}$ admits a circle-packing metric with zero discrete Gaussian curvatures, and its $1-$skeleton is VEL-hyperbolic.

\end{cor}
%\begin{rem}
%    The existence of a zero-curvature metric under the assumption of non-positive curvature was previously established for finite triangulations of compact surfaces; see e.g. \cite{2003Combinatorial}. For infinite triangulations, we prove our result using the parabolic method, specifically the CRF, while a proof via elliptic methods appears to be unknown. 
%\end{rem}
Here ``VEL'' refers to the \textbf{vertex extremal length}, which we will introduce in Section \ref{vel}. For a graph with bounded degree, it is VEL-parabolic (hyperbolic resp.) if and only if the simple random walk on it is recurrent (transient resp.), for which the bounded-degree condition cannot be dropped, see \cite[Theorem 8.2]{He_schramm}.

After that, we consider the case where the triangulation is the hexagonal triangulation $\mathcal{T}_H=(V_H,E_H,F_H)$ as shown in Figure \ref{regular}. There are known examples of locally univalent circle packings, e.g. the regular circle packing shown in Figure \ref{regular} and Doyle spirals, see \cite[Appendix C]{Stephenson_intro} and \cite{Doyle_spirals}.   
We prove the convergence result of the CRF provided that the initial metric is a small perturbation of the regular circle packing. 

\begin{thm}[Convergence of Euclidean CRF on $\mathcal{T}_H$]\label{mainthm} There is a universal constant $\epsilon_0>0,$ such that for any initial data $r(0)$ with  $\|\ln ({r(0)})\|_{l^2}\leq \epsilon_0,$  the solution $r(t)$ to the flow \eqref{unnormalized} converges to the regular circle-packing metric
\[
r_{reg}\equiv1
\]
as shown in Figure \ref{regular}. Moreover, the logarithmic of the flow $u(t)=\ln(r(t))$ satisfies $u\in C^1([0,\infty),l^2(H)),$ 
\begin{align*}
       \EE(u(t)):=\sum_{\{i,j\}\in E}|u(i,t)-u(j,t)|^2\rightarrow 0,\quad \mathrm{and}\quad \|u(t)\|_{l^\infty}\rightarrow 0,\quad t\to\infty.
    \end{align*}
%In particular, we have $\|u(t)\|_{l^\infty}\rightarrow 0.$
\end{thm}

\begin{figure}[htbp]
\centering
\includegraphics[scale=0.60]{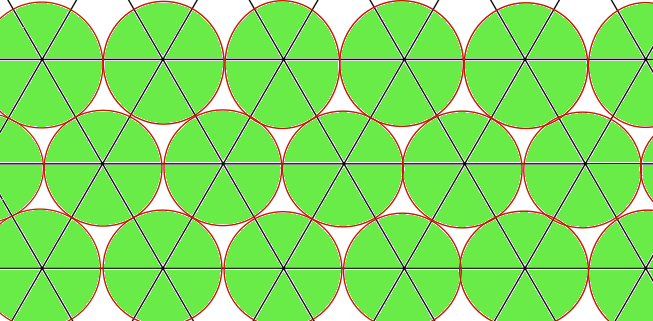}
%\caption{B}
\captionof{figure}{\small A part of a regular hexagonal packing.}
  \label{regular}
\end{figure} 
This result is based on our new observation of the special structure of the equation  \eqref{unnormalized} on the hexagonal triangulation and the energy estimates; see Section \ref{sec:5}.

From  Theorem \ref{converge_hyp} and Theorem \ref{mainthm}, we observe that under some assumptions on the initial data, the Euclidean CRF and hyperbolic CRF on infinite disk triangulations converge to circle-packing metrics with zero discrete Gaussian curvatures. Since for the cases of finite graphs the existence of circle-packing metrics with zero discrete Gaussian curvatures and the convergence of the CRF are equivalent, see e.g. \cite{2003Combinatorial,MR4334399}, we have the following conjecture of the uniformization of the CRF. 
\begin{conjecture}[Uniformization of the CRF]\label{conj}
    Let $\mathcal{T}$ be an infinite disk triangulation with $\Phi\in [0,\frac{\pi}{2}]^E$, then the following statements are equivalent.
\begin{enumerate}[(A)]
    \item The CRF on $\mathcal{T}$ converges for any initial value in hyperbolic background geometry.
    \item There exists a locally finite univalent circle packing of $\mathcal{T}$ in the unit disc.
    \item The $1-$skeleton of $\mathcal{T}$ is VEL-hyperbolic.
\end{enumerate}
\end{conjecture}
We expect the CRF serve as a new tool to study  circle packing problems for infinite triangulations.

\subsection{Organization of the paper and proof strategies}

Firstly, we prove the long-time existence of the flow \eqref{infinite_flow} in the spirit of the work of Shi \cite{Shi_noncompact}. We choose an exhaustive sequence of simply connected finite triangulations $\{\mathcal{T}_j\}_{j=1}^\infty$ such that $\cup_j\mathcal{T}_j=\mathcal{T}$ and $\mathcal{T}_j\subset\mathcal{T}_{j+1}$ for $j\geq 1.$ For each $\mathcal{T}_j,$ we construct a solution  $u^{[j]}(t)$ of the CRF with proper Dirichlet boundary condition. After that, we establish a maximum principle for the CRF on $\mathcal{T}_j$, inspired by Chow and Luo \cite{2003Combinatorial}, obtain some a priori estimates, and prove that the flow $u^{[j]}(t)$ has a convergent subsequence, which will converge to a solution $u\in C^{\infty}(V\times[0,\infty);\mathbb{R})$ with any initial value $r(0)$. The uniform boundedness of the discrete Gaussian curvature obtained in Proposition~\ref{prop:degree} plays an essential role in the proof.

Secondly, we prove the uniqueness of the flow \eqref{infinite_flow} given that the discrete Gaussian curvature $K$ is uniformly bounded on $V\times [0,T]$. We observe that the difference of two solutions to the flow \eqref{infinite_flow} can be written as a heat equation on the triangulation $\mathcal{T}$ with a time-changing edge weight $\omega(t)\in\mathbb{R}_+^E$ satisfying, for any $i\in V,$ \begin{align}\label{esti_weight}
    \sum_{j\in V:j\sim i}\omega_{ij}(t)\le C,\quad \forall t\leq T,
\end{align}
where  $C=C(\|K\|_{L^\infty(V\times [0,T])},T).$ The uniform boundedness follows from a key estimate of angle derivatives in Lemma~\ref{est_deri}. 
Then the uniqueness follows from a maximum principle, Lemma~\ref{mp2}, for parabolic operators with bounded time-changing weights, inspired by Wu \cite{wu1993ricci}.

Moreover, in Section \ref{sec:4}, we prove the convergence of the flow \eqref{infinite_flow} in hyperbolic background geometry. Recall that Chow and Luo's convergence result for the CRF on a finite triangulation of a compact surface \cite{2003Combinatorial} is derived from the variational structure of the flow, as the CRF represents the negative gradient flow of a convex functional. This approach has been widely adopted in subsequent works on flow methods; see, e.g. \cite{luo2004combinatorial,MR3807319,MR3825607,MR4024520,ge2021combinatorial,MR4334399,MR4466650}. However, in our setting of infinite triangulations, no well-defined functional is known to serve this role. Instead, we establish the convergence result using the maximum principle for curvature and the intrinsic geometric properties of the hyperbolic setting. In particular, we prove that the non-positivity of the Gaussian curvature is preserved by the CRF, which yields the monotonicity of the solution to the flow. This guarantees the existence of the limit metric with vanishing Gaussian curvature. After that, we prove some existence results of singular circle-packing metrics with prescribed discrete Gaussian curvatures with the flow approach.

Finally, in proving Theorem \ref{mainthm}, we introduce the analysis on the hexagonal triangulation $H$  in Section \ref{sec:5}, and observe that the flow \eqref{infinite_flow} can be reformulated as a semilinear parabolic equation on $H,$ for any $v\in H,t\geq 0,$ 
\begin{align*}
    \ddt{u_v}-\Delta_{\omega_*} u_v=F(Du)(v),
\end{align*}
where $\Delta_{\omega_*}$ is a constant discrete Laplacian operator defined in Section \ref{sec:5}, and $F:\mathbb{R}^6\rightarrow\mathbb{R}$ is an analytic function satisfying $F(0)=0$ and $\nabla F(0)=0$. Then the result follows from the analysis of asymptotic behaviors of this equation via energy estimates in Section \ref{sec:5}.

\section{Preliminaries}\label{sec:2}

\subsection{Circle-packing metrics and discrete conformal factors}

Let $\Sigma$ be a surface, and $\mathcal{T}=(V,E,F)$ be a triangulation of $\Sigma,$ where $V,E$ and $F$ are the sets of vertices, edges, and faces respectively. In this article, we always assume that the triangulation $\mathcal{T}$ is locally finite.
For simplicity, we use $\{i\}_{i=1}^{|V|}$ to denote the vertex set $V$, where $|V|$ is the cardinality of the set $V.$  $\mathcal{T}$ is called an infinite triangulation if $|V|=\infty.$ We denote by $i\sim j$ that two vertices $i$ and $j$ are connected by an edge in $E$.
A positive function $r\in \mathbb{R}_+^V$ defined on the vertex set $V$ is called a \textbf{circle-packing metric}. A function $\Phi\in [0,\frac{\pi}{2}]^E$ is called the intersection angle.  Now fixing a pair $(\mathcal{T},\Phi),$ we assign to each edge $e=\{i,j\}\in E$  the length 
\begin{align}\label{edgelength}
l_{ij}=\sqrt{r_i^2+r_j^2+2\cos{\Phi_{ij}}r_ir_j}
\end{align}
in Euclidean background geometry and 
\begin{align}\label{edgelength_hyp}
    l_{ij}=\cosh^{-1}({\cosh{r_i}\cosh{r_j}+\cos{\Phi_{ij}}\sinh{r_i}\sinh{r_j}})
\end{align}
in hyperbolic background geometry.
A face with vertices $i,~j$ and $k$ is shown in Figure \ref{cp_metric}. Gluing all triangles along common edges, one can obtain a polyhedral surface given by $(\mathcal{T},\Phi)$ and $r.$
\begin{figure}[htbp]
\centering
\includegraphics[scale=0.50]{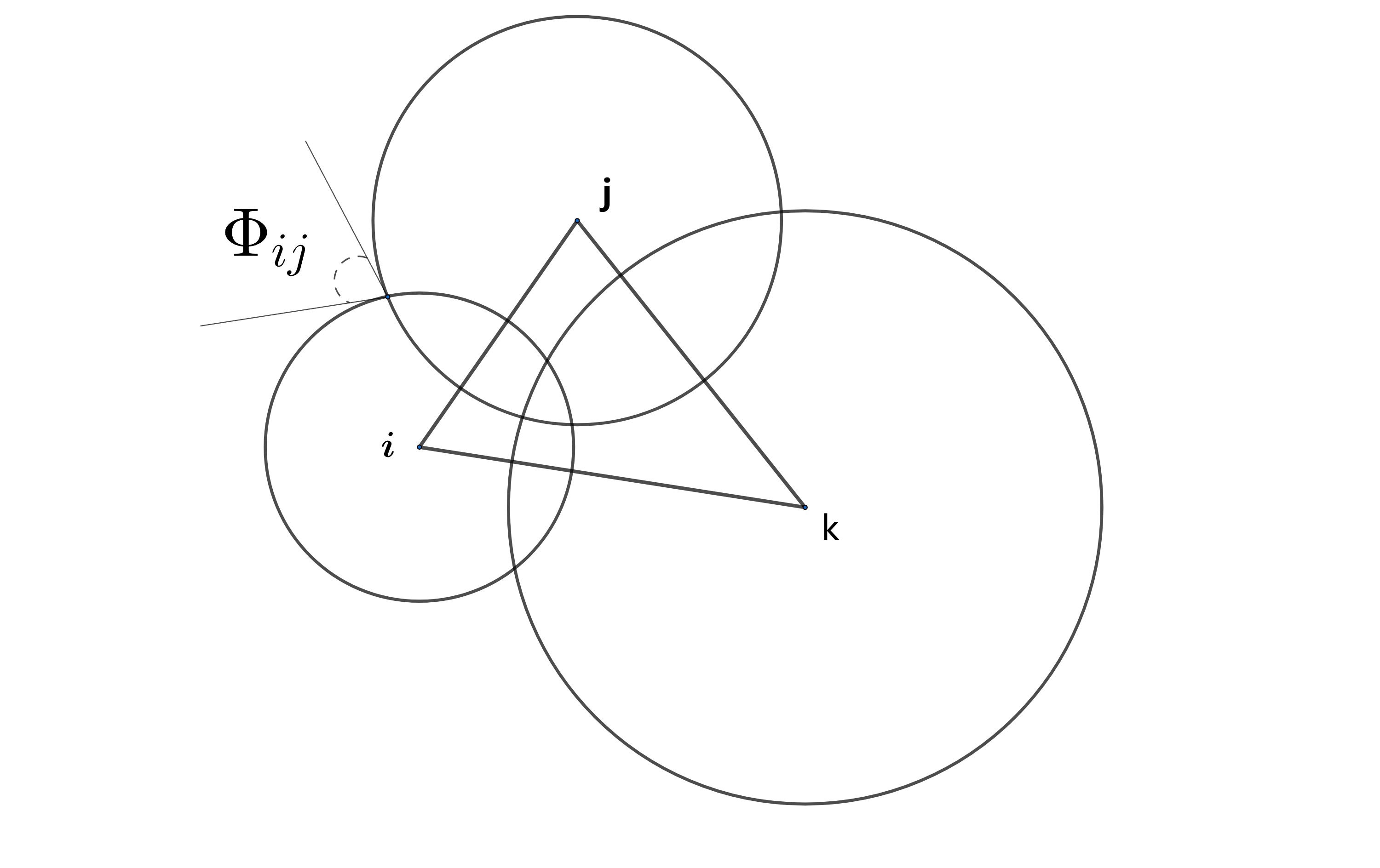}
%\caption{B}
\captionof{figure}{\small The circle-packing metric in Euclidean background geometry on a triangle.}
  \label{cp_metric}
\end{figure} 

Given a pair $(\mathcal{T},\Phi)$, all inner angles of triangles in $\mathcal{T}$ are determined by the circle-packing metric $r.$ The \textbf{discrete Gaussian curvature} $K_i$ at the vertex $i$ is defined as 
\[
K_i=2\pi-\sum_{\{i,j,k\}\in F}\theta_i^{jk},
\]
where $\theta_i^{jk}$ is the inner angle at the vertex $i$ of the face $\{i,j,k\}.$
By change of variables, let $u_i=\ln r_i~(\ln{\tanh{\frac{r_i}{2}}}$ resp.) in Euclidean (hyperbolic resp.) background geometry, flows \eqref{unnormalized} and \eqref{hyperbolic_flow} can be written in a uniform way as below 
\begin{align}
&\ddt{u_i}=-K_i,~~\forall~i\in V.\label{infinite_flow}
\end{align}
The new variable $u=(u_i)_{i\in V}\in\mathbb{R}^V((-\infty,0)^V~\text{resp.})$ is called the \textbf{discrete conformal factor} of the circle-packing metric $r$ in Euclidean (hyperbolic resp.) background geometry in the literature; see \cite{gu2008computational}. 

\subsection{Variational principles}
Now we introduce the variational principle for circle packings in Euclidean and hyperbolic background geometries, which combined the work of Thurston \cite{thurston1980geometry} and Colin de Verdi\`ere \cite{MR1106755}.
\begin{lem}\label{vari}
    Let $f=\{1,2,3\}$ be a triangle with a circle-packing metric in Euclidean or hyperbolic background geometry. We assume $\Phi\in[0,\frac{\pi}{2}]^3.$  Let $(r_i)_{i=1,2,3}$ and $(u_i)_{i=1,2,3}=(\ln r_i)_{i=1,2,3}$($\ln\tanh{\frac{r_i}{2}}$ resp.) be radii and the discrete conformal factors of the circle-packing metric in Euclidean (hyperbolic resp.) background geometry. Moreover, let $\theta_1^{23},\theta_2^{31}$ and $\theta_3^{12}$ be inner angles of the triangle $f$ at the vertices $1,2$ and $3$, respectively. The following statements hold.
    \begin{enumerate}[I)]
    \item  $\pp{\theta_i^{jk}}{u_j}>0$,~  $\forall i\neq j,$ and $\pp{\theta_i^{jk}}{u_i}<0,~\forall i=1,2,3.$
        \item  $  \pp{\theta_i^{jk}}{u_j}=\pp{\theta_j^{ik}}{u_i},  ~~\forall\{i,j,k\}=\{1,2,3\}. $
        \item $\pp{\theta_i^{jk}}{u_1}+\pp{\theta_i^{jk}}{u_2}+\pp{\theta_i^{jk}}{u_3}=0(<0~ \text{resp.}) ~\text{in Euclidean (hyperbolic resp.) background geometry},
        \forall i\in\{1,2,3\}.$
    \end{enumerate}
\end{lem}
The proof of Lemma \ref{vari} is straightforward via direct calculation. By Lemma \ref{vari}, for two adjacent vertices $i$ an $j$ in a triangulation $\mathcal{T}$ with a circle-packing metric,
\begin{align}\label{change}
\pp{K_i}{u_j}=\pp{K_j}{u_i}<0,\forall i\sim j, ~\text{and} \ \pp{K_i}{u_i}>0.
\end{align}
Moreover, 
$$ \pp{K_i}{u_i}=-\sum_{j\sim i}\pp{K_i}{u_j}-\lambda B_i,
 $$ where $B_i=B_i(u)$ is a positive function, and $\lambda=0$ or $-1$ corresponds the Euclidean or hyperbolic setting. The next proposition follows directly from the definition.

 \begin{prop}\label{prop:degree}
    For a triangulation $\mathcal{T}=(V,E,F)$ with a circle-packing metric, for any $i\in V,$ $$2\pi-\deg(i)\pi\leq K_i(u)<2\pi.$$ For a triangulation $\mathcal{T}=(V,E,F)$ with $\sup_{i\in V}\deg(i)\leq N<\infty,$ for any circle-packing metric, 
    $$\sup_{i\in V}|K_i(u)|\leq N\pi.$$
 \end{prop}

For detailed estimates of partial derivatives $\{\pp{\theta_i^{jk}}{u_i}\}_{\{i,j,k\}=\{1,2,3\}}$, we have the following result, which extends the estimate of He in \cite{HE}.
\begin{lem}[Estimate of the partial derivatives]\label{est_deri}
    Let the face $f$ and the partial derivaties $\{\pp{\theta_i^{jk}}{u_j}\}_{\{i,j,k\}=\{1,2,3\}}$ be defined as above. Then
    \begin{align*}
        \pp{\theta_i}{u_j}\le C\theta_i,~~\forall i\sim j,
    \end{align*}
for some uniform constant $C$ in both Euclidean and hyperbolic background geometries.
\end{lem}
We will leave the proof of this lemma in Appendix \ref{appendix}.

\subsection{Vertex extremal length}\label{vel}
The vertex extremal length is a discrete analog of the extremal length of curve families in Riemann surfaces, which was used for studying square tilings by Schramm \cite{MR1244661}.

Let $G=(V,E)$ be a connected infinite graph, and $A\subset V$ be a finite subset of $V$. We denote by $\Gamma(A,\infty)$ the family of infinite paths starting from the vertex set $A$, which cannot be contained in any finite subset of $V$. Let $m:V\rightarrow[0,\infty)$ be a function on $V$. For an infinite path $\gamma=v_0v_1\cdots v_n\cdots$, the integral of $m$ over $\gamma$ is defined by 
\[
\int_\gamma \mathrm{d}m=\sum_{i=1}^\infty m(v_i).
\]

We call $m$ a $\Gamma(A,\infty)$-admissible function if 
\[
\int_\gamma\mathrm{d}m\ge 1,~\forall \gamma\in \Gamma(A,\infty).
\]
The vertex extremal length of $\Gamma(A,\infty)$ is defined as 
$$\mathrm{VEL}(A,\infty)=\max\left\{\frac{1}{\|m\|^2_{l^2}}:m:V\rightarrow\infty~\text{is }~\Gamma(A,\infty)\text{-admissble}\right\}.$$
\begin{defn}
    An infinite graph $G=(V,E)$ is called VEL-parabolic if  there exists a finite vertex set $A$ such that 
    \[
    \mathrm{VEL}(A,\infty)=\infty.
    \]
    Otherwise, $G$ is called VEL-hyperbolic.
\end{defn}

Finally, we introduce some basic definitions in graph theory:

\textbullet $d(i,j)$: the combinatorial distance from $i$ to $j$, given by
\[
d(i,j)=\inf\{d\in\mathbb{N}:\text{there exists a path $i_0i_1...i_d$ connecting}~i ~\text{and}~j~\text{with}~i_0=i,i_d=j\}.
\]

\textbullet $B_n(i)$: the ball on the graph with radius $n$ centered at the vertex $i$, that is $$B_n(i)=\{j\in V:d(i,j)
\le n\},$$

\textbullet $\deg(i)$: the number of edges that are adjacent to the vertex $i.$

\section{Well-posedness of CRF on infinite disk triangulations}\label{sec:3}

In this section, we study the well-posedness of the CRF on infinite disk triangulations.

\subsection{The long-time existence of the flow \eqref{infinite_flow}}
For the CRF on a finite triangulation, the long-time existence is guaranteed by the classical ODE theory and the maximum principle obtained by Chow and Luo \cite{2003Combinatorial}. In order to prove the long-time existence result for the CRF on infinite triangulations, we employ approximation techniques. We first establish the long-time existence of the solution to the CRFs in both Euclidean and hyperbolic background geometry. The uniqueness of the flow is contained in next subsection.

\begin{thm}\label{longtime_eu}
 For circle-packing metrics in Euclidean (hyperbolic resp.) background geometry, let $\mathcal{T}=(V, E, F)$ be an infinite disk triangulation, and let $\Phi\in [0,\frac{\pi}{2}]^E$. Then for any initial data $r(0)\in \R_+^V,$ there exists a global solution $r\in C^{\infty}_t(V\times[0,\infty))$ with $u(t)\in \R_+^V$ for any $t\geq 0$ to the flow \eqref{unnormalized} (\eqref{hyperbolic_flow} resp.).
\end{thm}
%\begin{thm}\label{longtime_hyp}

%For circle-packing metrics in hyperbolic background geometry, let $\mathcal{T}=(V, E, F)$ be an infinite disk triangulation, and let $\Phi\in [0,\frac{\pi}{2}]^E$.  Then there exists a global solution $u\in C^{\infty}_t(V\times[0,\infty))$ to the flow \eqref{hyperbolic_flow} with any initial value. 
%\end{thm}

Let $\mathcal{T}=(V,E,F)$ be a locally finite infinite triangulation of $\mathbb{R}^2$. Let $\mathcal{T}_i$ be the simply connected subcomplexes of $\mathcal{T}$ consisting of triangles whose vertices are contained in $B_i(1)$. Then they satisfy
\[
\mathcal{T}_i\subset\mathcal{T}_{i+1},\ \forall i\geq 1,~\quad\cup_i\mathcal{T}_i=\mathcal{T}.
\]
We denote by $V_i,~E_i$, and $F_i$ the sets of vertices, edges, and faces of $\mathcal{T}_i$. Set $\partial V_i$ and $\mathrm{int}(V_i)$ the boundary vertices and inner vertices of $V_i$. Recall that for Euclidean (hyperbolic resp.) background geometry, we always use $u_i=\ln r_i~(\ln{\tanh{\frac{r_i}{2}}}$ resp.) and require that $u\in \R^V$ ($u\in (-\infty,0)^V$ resp.) Let $u(0)$ be a discrete conformal factor. We consider the CRF on $\mathcal{T}_i$ as follows
\begin{eqnarray}\label{finite_flow}
   \left\{
   \begin{aligned}
   &\ddt{u_j^{[i]}(t)}=-K_j,~\forall~j\in\mathrm{int}(V_i),~\forall t> 0. \\
   &u^{[i]}_j(t)=u_j(0),~\forall (j,t)\in (V_i\times\{0\})\cup(\partial V_i\times(0,\infty)).
   \end{aligned}
   \right.
   \end{eqnarray}
Since $V_i$ is finite, the local existence of the above equations follows from Picard's theorem. The long time existence of those flows \eqref{finite_flow} is similar to that of compact surfaces, see \cite{2003Combinatorial}. 
First, let us introduce the maximum principle for the flow $u^{[i]}(t)$.

\begin{lem}[Maximum principle I]\label{max_principle}Let $u^{[i]}(t)=\ln{r^{[i]}}(t)$ ($u^{[i]}(t)=\ln\tanh{\frac{r^{[i]}(t)}{2}}$ resp.) be a solution to the equation \eqref{finite_flow} in Euclidean (hyperbolic resp.) background geometry. Let $K_j(u^{[i]}(t))$ be the discrete Gaussian curvature of the vertex $j\in\mathrm{int}(V_i).$ Set $m(t)=\min_{j\in \mathrm{int}(V_i)}K_j(u^{[i]}(t))$ and $M(t)=\max_{j\in \mathrm{int}(V_i)}K_j(u^{[i]}(t))$.
    We have $m_*(t):=\min(m(t),0)$ is non-decreasing, and $M_*(t)=\max(M(t),0)$ is non-increasing with respect to $t$.\end{lem}
    \begin{rem}
    In \cite{2003Combinatorial}, Chow and Luo first proved the maximum principle of the CRF on compact surfaces without boundaries.
    \end{rem}

\begin{proof}
        Let $G_i$ be the induced subgraph on the set of vertices $\mathrm{int}(V_i),$ for which $i$ is fixed.  By the CRF equation, for any $j\in \mathrm{int}(V_i),$
        \begin{align}
    \ddt{K_j(u^{[i]}(t))}&=-\sum_{k\sim j,k \in V_i}\pp{K_j}{u_k}\ddt{u^{[i]}_k}-\pp{K_j}{u_j}\ddt{u^{[i]}_j}\nonumber
    \\&=-\sum_{k\sim j,k \in\mathrm{int}(V_i)}\pp{K_j}{u_k}K_k(u^{[i]}(t))-\pp{K_j}{u_j}K_j(u^{[i]}(t)).
            \label{eq:curveq}
\end{align}  Set $f(j,t):={K_j(u^{[i]}(t))}.$  Then by Lemma~\ref{vari} and \cite[Proposition~3.2]{2003Combinatorial}, we have
$$\ddt f(j,t)=\sum_{k\in \mathrm{int}(V_i): k\sim j} C_{jk}(f(k,t)-f(j,t))+B_jf(j,t),$$ where $C_{jk}=C_{kj}$ are positive functions, and $B_j$ is a non-positive function. Then the results follow from the well known maximum principle for linear parabolic equations on $G_i.$
        
    \end{proof}

Thanks to Lemma \ref{max_principle},  we have the following corollary.
\begin{cor}\label{uniform1}
    Let $u^{[i]}(t)$ be a solution to \eqref{finite_flow} in Euclidean or hyperbolic background geometry. Then
    \begin{align}
          |K_j(u^{[i]}(t))|\le \max\left\{\left|\min_{j\in \mathrm{int}(V_i)}K_j(u(0))\right|,2\pi\right\},\forall j\in  \mathrm{int}(V_i),~\forall t\ge 0.
    \label{uniform2}  
    \end{align}

\end{cor}
\begin{proof}
    By the definition of $K_j$, one easily sees that $K_j(u^{[i]}(t))\le 2\pi.$ Therefore, by Lemma \ref{max_principle}, we have 
    \[
    K_j(u^{[i]}(t))\ge\min\left\{\min_{j\in \mathrm{int}(V_i)}K_j(u(0)),0\right\},~\forall j\in \mathrm{int}(V_i),
    \]
This proves the result.
\end{proof}

We need the following proposition in \cite[Lemma 3.5]{2003Combinatorial} for the proof of the long-time existence of CRF in hyperbolic background geometry.
\begin{prop}\label{hyp_char}
    For a hyperbolic triangle $\{i,j,k\}$ with a circle-packing metric, let $u_i,u_j$ and $u_k$ be the corresponding discrete conformal factors. Then for any $\epsilon>0,$
    there exists a number $\delta<0$ such that if $\delta\le u_i<0,$ we have
    \[
    \theta_{i}^{ik}(u_i,u_j,u_k)\le \epsilon.
    \]
    \end{prop} 
    Now we are ready to prove the long-time existence of the CRF.

\begin{proof}[\textbf{Proof of Theorem \ref{longtime_eu}.}]
We first prove the result for the Euclidean background geometry. 
 Fix a vertex $j\in V,$ and consider sufficiently large $i$ such that $j\in \mathrm{int}(V_i).$ Fix the time interval $[0,T]$ for $T>0.$ By the definition of $K_j,$
\[
|K_j(u^{[i]}(t))|\le(2+\deg(j))\pi,~\forall t\in[0,T].
\] Hence, by the CRF equation, $$|u^{[i]}_j(t)|\le |u_j(0)|+(2+\deg(j))\pi T.$$
By the CRF equation,  $$\sup_i\|u^{[i]}_j(t)\|_{C^1[0,T]}\leq C(u_j(0),\deg(j),T).$$
Moreover, by the evolution equation for the curvature \eqref{eq:curveq} and the uniform bound of $\sup_k\|u^{[i]}_j(t)\|_{C^1[0,T]}$ for $k\sim j,$
$$\sup_i\|u^{[i]}_j(t)\|_{C^2[0,T]}\leq C.$$ Then by the Arzel\`a-Ascoli theorem and the standard diagonal argument, there is a subsequence $\{u^{[i_l]}(t)\}_{l=1}^\infty$ of $\{u^{[i]}(t)\}_{i=1}^{\infty}$ such that for each vertex $j\in V$, $u^{[i_l]}_j(t)$  converges in $C^1[0,T]$ to some  $u^*_j(t)$ as $l\to \infty,$ which satisfies the equation \eqref{infinite_flow}. 

Since we obtain the solution $u^*$ to the time interval $[0,T]$ for any $T>0,$  we obtain a global solution of the flow \eqref{infinite_flow}. Moreover, by the smoothness of $K(u)$ with respect to $u$, the flow $u^*(t)$ is smooth with respect to $t$.  This proves the result for the Euclidean case.

Now we consider the case of hyperbolic background geometry.  In this case, the configuration space is given by $u\in (-\infty,0)^V.$ We only need to prove that the solution $u^*$ we constructed stays away from zero for any vertex, i.e. there is no vertex $j\in V$ and finite $T>0,$ such that $\lim_{t\to T}u_j^*(t)=0.$ 
 
 Fix a vertex $j\in V,$ for sufficiently large $i,$ we assume that $j\in \mathrm{int}(V_i).$  By Proposition \ref{hyp_char}, there is a constant $\delta_j<0$ such that if $u_j^{[i]}(t)\geq \delta_j$, then
   \begin{align}\label{fact2}
          \theta_j^{kl}(u^{[i]}(t))< \frac{2\pi}{\deg(j)},~~\text{whenever}~\{j,k,l\}~\text{is a face.}
   \end{align}
  We claim that \begin{align}\label{hyperconstant}\sup_{[0,T]}u^{[i]}_j(t)\leq \bar{\delta}_j, \quad \mathrm{with}\quad
     \bar{\delta}_j:=\frac{1}{2}\max\{u_j(0),\delta_j\}.
   \end{align}
  Suppose that this is not true, then $A:=\{t\in[0,T]:u^{[i]}_j(t)>\bar{\delta}_j\}\neq \emptyset.$ Let $t_0:=\inf A.$ Then $t_0>0.$ 
   Then by \eqref{fact2}, 
   $$0\leq \ddt{u^{[i]}_j(t_0)}=-K_j(u^{[i]}(t_0))=\sum_{\{j,k,l\}\in F}\theta_{j}^{kl}(u^{[i]}(t_0))-2\pi<0,$$
  which leads to a contradiction and proves the claim. The result for the hyperbolic case follows from the claim.
\end{proof}

By Theorem \ref{longtime_eu} and Corollary \ref{uniform1}, we have the following result.
\begin{prop}\label{bounded_initial}
    Let $u_0$ be a discrete conformal factor in Euclidean or hyperbolic background geometry such that the corresponding discrete Gaussian curvatures are uniformly bounded.
    Then there exists a global solution $u(t)$ of the flow \eqref{infinite_flow} with initial value $u_0$ such that the discrete Gaussian curvature $K(u(t))$ is uniformly bounded on $V\times[0,\infty).$
\end{prop}
\subsection{The uniqueness of the flow}
For the uniqueness of  the solution to the CRF on infinite triangulations, we have the following theorem.
\begin{thm}\label{uniqueness}
    Let $u(t)$ and $\hat{u}(t)$ be two solutions to the flow \eqref{infinite_flow} for $t\in[0,T]$ in Euclidean or hyperbolic background geometry with the same initial value $u(0)=\hat{u}(0)$. If $K(u(t))$ and $K(\hat{u}(t))$ are uniformly bounded on $V\times[0,T]$, then $u\equiv\hat{u}.$
\end{thm}
In order to prove Theorem \ref{uniqueness}, we first introduce a discrete Laplacian operator on the triangulation $\mathcal{T}$. 
For a positive function on the edge set $\omega\in\mathbb{R}_+^{E},$ we call it an edge weight function. We define the discrete Laplacian $\Delta_{\omega}$ associated to the edge weight $\omega$ as, for any $g\in\mathbb{R}^V$,
\begin{equation}\label{eq:lap}
\Delta_\omega g_i=\sum_{j:j\sim i}\omega_{ij}(g_j-g_i), ~\forall i\in V,
\end{equation}
which is the standard Laplacian on a weighted graph; see e.g. \cite{chung1997spectral,MR3822363}.

%%%%这里详细的写一下
To prove the uniqueness of the flow \eqref{infinite_flow}, we introduce a maximum principle on infinite triangulation $\mathcal{T}$, which is a discrete analog to the maximum principle of Ricci flow on the entire plane $\mathbb{R}^2$ introduced by Wu \cite{wu1993ricci}. 

\begin{lem}[Maximum principle II]\label{mp2}
Let $\mathcal{T}$ be an infinite triangulation, and let $\{\omega(t)\}_{t\ge 0}$ be a one-parameter family of weights on $E$ for $t\in [0,T]$ with $T>0$ such that
\begin{align}\label{weight_assum}
\sum_{j:j\sim i}\omega_{ij}(t)\le C,~~\forall(i,t)\in V\times[0,T],
\end{align}
where $C$ is a uniform constant. Suppose a function $g:V\times[0,T]\rightarrow\mathbb{R}$ satisfies 
    \begin{align}\label{mp3}
        \ddt{g}\le\Delta_{\omega(t)}g+hg.
    \end{align}
    If $g$ is a bounded function in $V\times[0,T]$ with $g\le 0$ at $t=0$, and if $h\le B$ for some constant $B$, then $$g(j,t)\le 0,~\quad\forall(j,t)\in V\times [0,T].$$
\end{lem}
\begin{proof}
    Without loss of generality, we may assume that $B\le0.$ In fact, by setting $\bar{g}=e^{-Bt}g$ and $\bar{h}=h-B,$ the result is reduced to the above case. 
    Fixing a vertex $j_0\in V$, we denote by $d_0(j)=d(j,j_0)$ the combinatorial distance function to the vertex $j_0.$ By \eqref{weight_assum}, we have
    \begin{align*}
        \Delta_{\omega(t)}d_0(j)\le \sum_{k:k\sim j}\omega_{jk}(t)\le C ,\ \forall j\in V.
    \end{align*}
  For $\delta>0,$ set ${g}_\delta:={g}-\delta d_0-(C+1)\delta t.$ Then we have 
    \begin{eqnarray*}
        \ddt{{g}_\delta}&\le&\Delta_{\omega(t)}{g}_\delta+\delta(\Delta_{\omega(t)}d_0-C-1)+{h}{g}_\delta+{h}(\delta d_0+C\delta t+\delta t)\\
        &<& \Delta_{\omega(t)}{g}_\delta+{h}{g}_\delta.
    \end{eqnarray*}

We claim that ${g}_\delta\le 0 $ on $V\times[0,T]$.  
Suppose that this is not true, i.e. $\sup_{V\times [0,T]}g_\delta >0.$ Since ${g}$ is bounded, $$\sup_{t\in[0,T]}g_\delta(j,t)\to-\infty,\quad j\to \infty.$$ Moreover, $g_\delta(j,0)\leq g(j,0)\leq0,$ for any $j\in V.$ Then there exists $(j_1,t_1)\in V\times(0,T]$ attaining the maximum of $g_\delta$ on $V\times [0,T].$
Hence, $\ddt{{g}_\delta}(j_1,t_1)\ge 0$ and $\Delta_{\omega(t_1)}{g}_\delta(j_1,t_1)\le 0.$ Therefore,
\[
0\le\ddt{{g}_\delta}(j_1,t_1)< \Delta_{\omega(t_1)}{g}_\delta(j_1,t_1)+{h}{g}_\delta(j_1,t_1)\leq0,
\]
which leads to a contradiction. This proves the claim.

Finally, letting $\delta\rightarrow0$, we prove the result.
\end{proof}
Then we have the following corollary.
\begin{cor}\label{unique_lemma}
If $g(t)$ is a bounded solution to the equation
    \begin{align}\label{unique1}
        \ddt{g(t)}=\Delta_{\omega(t)}g + hg
    \end{align}\label{infinite_unique}
with $g(0)\equiv 0$ on $V$, where $\omega(t)$ satisfies  \eqref{weight_assum}, and the function $h\le B$ for some constant $B$, then for any $t\geq 0,$  $g(t)\equiv 0.$

\end{cor}
For the uniqueness of the flow \eqref{infinite_flow} in both Euclidean or hyperbolic background geometry, we need the following lemmas.
\begin{lem}
    Let  $u$ and $\hat{u}$ be two solutions of the flow \eqref{infinite_flow}  in Euclidean or hyperbolic background geometry.  Then
    \begin{align}\label{unique_equ_hyp}
    \ddt{(u(t)-\hat{u}(t))}=\Delta_{\omega(t)}(u(t)-\hat{u}(t)) - h(u-\hat{u}),
\end{align} where
\begin{align}\label{weight1}
\omega_{ij}(t)=-\int_0^1\pp{K_i}{u_j}(su(t)+(1-s)\hat{u}(t))\mathrm{d}s>0,
\end{align}
\begin{align}\label{func1}
    h_i(t)=\int_0^1\left\{\pp{K_i}{u_i}(su(t)+(1-s)\hat{u}(t))+\sum_{k:k\sim i}\pp{K_i}{u_k}(su(t)+(1-s)\hat{u}(t))\right\}\mathrm{d}s\ge0.
\end{align}

\end{lem}
\begin{proof}By the flow \eqref{infinite_flow}  in Euclidean or hyperbolic background geometry, we have
\begin{align}\label{difference0}
    \ddt{(u_i(t)-\hat{u}_i(t))}=-(K_i(u(t))-K_i(\hat{u}(t))),~~\forall i\in V.
\end{align}
For a fixed time $t$, by the Newton-Leibniz formula and Lemma \ref{vari},
\begin{align*}
    &K_i(u(t))-K_i(\hat{u}(t))\\
    &=\sum_{j:j\sim i}\int_0^1\pp{K_i}{u_j}(su(t)+(1-s)\hat{u}(t))\mathrm{d}s\cdot[u_j(t)-\hat{u}_j(t)]+
    \\&\int_0^1\pp{K_i}{u_i}(su(t)+(1-s)\hat{u}(t))\mathrm{d}s\cdot[u_i(t)-\hat{u}_i(t)]
    \\&=\sum_{j:j\sim i}\int_0^1\pp{K_i}{u_j}(su(t)+(1-s)\hat{u}(t))\mathrm{d}s\cdot[(u_j(t)-\hat{u}_j(t))-(u_i(t)-\hat{u}_i(t))]+
    \\&\int_0^1\left\{\pp{K_i}{u_i}(su(t)+(1-s)\hat{u}(t))+\sum_{j:j\sim i}\pp{K_i}{u_j}(su(t)+(1-s)\hat{u}(t))\right\}\mathrm{d}s\cdot(u_i(t)-\hat{u}_i(t)).
\end{align*} This proves the lemma.

\end{proof}

   We need the following lemma to bound the angles when conformal factors vary.
\begin{lem}\label{compare_hyp}
    Let $f=\{i,j,k\}$ be a face with a circle-packing metric in Euclidean or hyperbolic background geometry. There exist universal constants $\epsilon_0$ and $C_0$ such that if discrete conformal factors $u_f:=(u_i,u_j,u_k)$ and $u'_f:=(u_i,u_j',u_k')$ satisfy $\max\{|u_j-u_j'|,|u_k-u_k'|\}\leq\epsilon_0,$ then 
   \begin{align*}
       \theta_i^{jk}(u_f)\le C_0\theta_i^{jk}(u_f').
   \end{align*}
\end{lem}
\begin{proof}
     Considering the function $H_i^{jk}(u)=\ln\theta_i^{jk}(u),$ for $\theta_{i}^{jk}\neq 0,$
    we have 
    \[
    \pp{H_i^{jk}}{u_j}=\frac{1}{\theta_i^{jk}}\pp{\theta_i^{jk}}{u_j}.
    \]
    By Lemma \ref{est_deri} and Lemma~\ref{vari}, $0< \pp{H_i^{jk}}{u_j}\le C.$ We have a similar bound for $\pp{H_i^{jk}}{u_k}.$
    Therefore, for $\max\{|u_j-u_j'|,|u_k-u_k'|\}\le \epsilon_0,$ 
\begin{eqnarray*}
    \ln\frac{\theta_{i}^{jk}(u_f)}{\theta_{i}^{jk}(u_f')}&=&\int_0^1\left\{\pp{H_i^{jk}}{u_j}(su_f+(1-s)u_f')(u_j-u_j')+\pp{H_i^{jk}}{u_k}(su_f+(1-s)u_f')(u_k-u_k')\right\}ds\\
    &\leq & 2C\epsilon_0.
\end{eqnarray*} Hence, $$\theta_i^{jk}(u_f)\le e^{2C\epsilon_0}\theta_i^{jk}(u_f'),$$ which proves the lemma.

\end{proof}

\begin{proof}[Proof of Theorem~\ref{uniqueness}]
   Let $u$ and $\hat{u}$ be two solutions of the flow \eqref{infinite_flow} in Euclidean or hyperbolic background geometry on $V\times [0,T]$ for $T>0$ with 
   \[u(t),\hat{u}(t)\in C_t^{\infty}(V\times [0,T]);K(u(t)),K(\hat{u}(t))\in L^\infty(V\times [0,T]).\]
 Let $$\mathcal{M}:=\|K(u)\|_{L^{\infty}(V\times [0,T])} +\|K(\hat{u})\|_{L^{\infty}(V\times [0,T])}<\infty.$$ Let $\omega(t)=\{\omega_{ij}(t)\}_{\{i,j\}\in E}$ be weights defined in \eqref{weight1}.
   We claim that there are constants $\delta=\delta(\mathcal{M})$ and  $C(\mathcal{M})$, such that $\sum_{j\sim i}\omega_{ij}(t)\le C(\mathcal{M})$ for all $(i,t)\in V\times[0,\delta_1],$ where $\delta_1:=\min\{\delta,T\}.$  Set $\delta:=\frac{\epsilon_0}{2\mathcal{M}},$ where $\epsilon_0$ is defined in Lemma \ref{compare_hyp}. 
Fix any $(i,t)\in V\times[0,\delta_1]$. Without loss of generality, we may assume that $u_i(t)\le\hat{u}_i(t).$ Note that by the CRF equation, $$\|u(t)-\hat{u}(t)\|_{L^{\infty}(V\times [0,\delta_1])}\leq \epsilon_0.$$ Then by Lemma \ref{compare_hyp} and  Lemma \ref{est_deri}, for any $j\sim i$ and $t\in[0,\delta_1],$ 
\begin{align}
&\left|\pp{K_i}{u_j}(su(t)+(1-s)\hat{u}(t))\right|\le C\sum_{k:\{i,j,k\}\in F}\theta_i^{jk}(su(t)+(1-s)\hat{u}(t))
\nonumber\\&\le  C\sum_{k:\{i,j,k\}\in F}\theta_i^{jk}(u_i(t),su_j(t)+(1-s)\hat{u}_j(t),su_k(t)+(1-s)\hat{u}_k(t))\nonumber
 \\&
\le CC_0\sum_{k:\{i,j,k\}\in F}\theta_i^{jk}(u(t)),
\end{align}
where the second inequality follows from the monotonicity of $\theta_i^{jk}$ with respect to the variable $u_i$ by I) in Lemma~\ref{vari}. Therefore
\[
\sum_{j:j\sim i}\omega_{ij}(t)\le 2CC_0\sum_{\{i,j,k\}\in F}\theta_i^{jk}(u(t))\le 2CC_0(2\pi+\mathcal{M}).
\] This proves the claim.
Let $g(t)=u(t)-\hat{u}(t).$ By \eqref{unique_equ_hyp}, we have 
\[
\ddt{g(t)}=\Delta_{\omega(t)}g(t)+\tilde{h}(t)g(t),
\]
where $\tilde{h}$ is a non-positive function. Moreover, $g_i(t)\le 2\delta_1\mathcal{M}$ for all $(i,t)\in V \times[0,\delta_1].$ 
Applying Corollary \ref{unique_lemma}, we have 
\[
u(t)\equiv \hat{u}(t),~\forall t\in[0,\delta_1].
\]
Since $\delta$ only depends on $\mathcal{M}$, by the same argument for $[n\delta,(n+1)\delta]$ for $n\in\mathbb{Z}_+$, we prove the uniqueness on $[0,T].$

\end{proof}
This yields the main result of the well-posedness of the CRF.
\begin{proof}[Proof of Theorem~\ref{wellposed-Euc}]
The result follows from Theorem~\ref{longtime_eu}
 and Theorem~\ref{uniqueness}.    
\end{proof}
%By Theorem \ref{longtime_eu}, \ref{longtime_hyp}, \ref{unique_equ} and \ref{unique_equ_hyp} we easily obtain Theorem \ref{wellposed-Euc} and \ref{wellposed-hyp}.

\section{Convergence of the hyperbolic CRF}\label{sec:4}
 In this section, we study the convergence of the CRF in hyperbolic background geometry. A planar graph $G$ is said to be $CP$ parabolic ($CP$ hyperbolic resp.) if
there is a locally finite circle pattern $P$ in the plane (unit disk resp.) with contact graph $G$. We recall the following combinatorial conditions.
\begin{enumerate}[(C1)]
    \item If the edges $e_1,e_2$ and $e_3$ form a null-homotopic closed curve in the triangulation $\mathcal{T}$, and if $\sum_{i=1}^3\Phi(e_i)\ge\pi,$ then these edges form the boundary of a triangle of $\mathcal{T}$.
    \item If the edges $e_1,e_2,e_3$ and $e_4$ form a null-homotopic closed curve in the triangulation $\mathcal{T}$ and $\sum_{i=1}^3\Phi(e_i)=2\pi,$ then these edges form the boundary of the union of two adjacent triangles of $\mathcal{T}$.
\end{enumerate}

For an infinite circle pattern on the plane, He proves the following important result.
\begin{thm}[He \cite{HE}]\label{he}
    Let $\mathcal{T}$ be a disk triangulation graph with intersection angles $\Phi\in [0,\frac{\pi}{2}]^E$ satisfying conditions (C1) and (C2). $\mathcal{T}$ is VEL-parabolic (VEL-hyperbolic resp.) if and only if $\mathcal{T}$ is $CP$ parabolic ($CP$ hyperbolic resp.).
\end{thm}

Given a VEL-hyperbolic triangulation $\mathcal{T}$, by Theorem \ref{he}, there exists a locally finite circle packing in the unit disc $U$. Therefore, a hyperbolic circle packing of $\mathcal{T}$ exists if we view the unit disc as the Poincar\'e disc model of the hyperbolic plane $\mathbb{H}^2$. However, He's method for constructing such circle packings relies on the embedding of the hyperbolic space into the plane $\mathbb{R}^2,$ which is hard to use for studying hyperbolic circle-packing metric with non-zero discrete Gaussian curvatures of infinite disk triangulations. 
Now we will prove Theorem \ref{converge_hyp} which provides an intrinsic way to find circle packings in hyperbolic space $\mathbb{H}^2$ for some triangulations. 
\begin{proof}[\textbf{{Proof of Theorem \ref{converge_hyp}}}]
    Let $u(0)$ be the discrete conformal factors such that \[
    K_i(u(0))\le 0,
    \]
    for each vertex $i.$ Let $u^*(t)$ be the flow \eqref{infinite_flow} obtained by the convergence of a subsequence of $\{u^{[i]}(t)\}_{i=1}^{\infty}$.
    By Lemma \ref{max_principle}, $K_j(u^{[i]}(t))\le 0$ for all $i\in\mathbb{N}$ and $j\in \mathrm{int}(V_i).$ Therefore, the limit metric satisfies $K_j(u^*(t))\le 0$ for all $(j,t)\in V\times[0,\infty).$
By the CRF equation, this means $u_j^*(t)$ is non-decreasing for each $j\in V$, and has a limit $u_j^*(\infty)\in (-\infty,\bar{\delta}_j]$ by the proof of Theorem~\ref{longtime_eu}, where $\bar{\delta}_j<0$ is given in \eqref{hyperconstant} depending on $u_j(0)$ and $\deg(j).$ 
Since $u_k^*(t)$ converges for all $k\in B_1(j),$ $K_j(u^*(t))$ must converge to some $K_j(\infty)$. Now we prove $K_j(\infty)=0$ for any $j\in V.$ 

Suppose it is not true, i.e. $K_j(\infty)<0$ for some $j\in V.$ Then there exists a $T>0$ such that $K_j(u^*(t))\le\frac{K_j(\infty)}{2}<0,$ whenever $t>T.$ Therefore, $\ddt{u^*_j}\ge-\frac{K_j(\infty)}{2}>0$ for $t>T.$ So $u_j^*(t)$ will exceed $\bar{\delta}_j$ within finite time, which leads to a contradiction. Therefore, $K_j(\infty)=0$ for all $j.$
\end{proof}
%%%%%%%%%%%%%%%%%%%%%%%%%%%%%%%%%%%%%%%%%%%%%%%%%%%%%%%%%%%%%%%%%%%%%%%%%%%%%%%%%%%%%%%%%%说如何导出推论！
Now we have a direct application of Theorem \ref{converge_hyp}.
\begin{proof}[Proof of Corollary \ref{negative}]
    Since $\mathcal{T}$ attains a circle-packing metric with non-positive discrete Gaussian curvature, by Theorem \ref{converge_hyp} there exists a circle-packing metric $r^*$ on $\mathcal{T}$ with zero discrete Gaussian curvatures. Applying the same argument of Stephenson, see e.g. \cite[Chapter 6 and Chapter 8]{Stephenson_intro}, there exists a locally finite maximal circle packing of $\mathcal{T}$ whose carrier is the unit disc. By Theorem \ref{he}, the triangulation is VEL-hyperbolic.
\end{proof}
\begin{rem}
    Although the argument of Stephenson in \cite{Stephenson_intro} does not deal with circle packings with intersection angles, one easily generalizes it to this setting. In addition, the word ``univalent'' here stands for the condition that the immersion of the surface with the circle-packing metric is an embedding.
\end{rem}

%We consider the following simple example.
%\begin{ex}
   
%\end{ex}
We have the following corollary for some special triangulations.
\begin{cor}
    Let $\mathcal{T}$ be an infinite disk triangulation with  $\Phi\equiv0.$ Suppose that $\deg(i)\ge 7$ for all vertices $i$, then $\mathcal{T}$ is VEL-hyperbolic, and it attains a hyperbolic circle-packing metric with zero discrete Gaussian curvature.
\end{cor}
\begin{proof}
     Taking a sequence of discrete conformal factors $u^{n}\equiv-n,$ by the Gauss-Bonnet Theorem, one easily sees that for any  $\epsilon>0,$ there exists a large number $N$ such that for any ~$n\ge N,$
     \[
     \theta_i^{jk}(u^n)>\frac{\pi}{3}-\epsilon,~\forall\{i,j,k\}\in F.
     \]
    Let $\epsilon=\frac{\pi}{21}.$  There exists $N$ such that  for any ~$n\ge N,$
     \[
     \theta_i^{jk}(u^n)>\frac{2\pi}{7},~\forall\{i,j,k\}\in F.
     \]
     Therefore, for any $i\in V,$
     $$K_i(u^{n})<2\pi-\frac{2\deg(i)\pi}{7}\le0.$$ The result follows from Theorem~\ref{converge_hyp}.\end{proof}
     
A natural generalization of the above discussions is to consider the combinatorial curvature flows with prescribed discrete Gaussian curvatures.
Let $\hat{K}\in (-\infty,2\pi)^V$ be prescribed discrete Gaussian curvatures on vertices. Then one considers the flow
\begin{align}\label{prescribed}
    \ddt{u_i}=-(K_i-\hat{K}_i),~~\forall i\in V.
\end{align}
All previous results extend to the setting of hyperbolic background geometry.
\begin{thm}\label{prescribe_conv}
    Given an infinite disk triangulation $\mathcal{T}$ with the intersection angle $\Phi\in[0,\frac{\pi}{2}]^E$,
    let $\hat{K}$ be a prescribed discrete Gaussian curvature. If $u(0)$ is a discrete conformal factor in hyperbolic background geometry such that $K_i(u(0))\le \hat{K}_i$ for each $i$, then there exists a solution $u(t)$ to the flow \eqref{prescribed} with initial value $u(0)$, which converges to a surface with prescribed discrete Gaussian curvatures $\hat{K}_i$.
Moreover, if the $K(u(0))$ is uniformly bounded, then there is a unique solution such that
   \[
 u=u(i,t)\in C_t^{\infty}(V\times[0,\infty)) ~\text{and} ~K=K(i,t)\in C_t^{\infty}(V\times[0,\infty))\cap L^\infty(V\times [0,\infty)).
 \]
\end{thm}
%For two discrete functions $g$ and $h$ in $\mathbb{R}^V$, we denote by $g\ge h$ if $g_i\ge h_i$ for all vertices $i.$ We have
\begin{thm}\label{parron}
    For an infinite disk triangulation $\mathcal{T}$ with the intersection angle $\Phi\in[0,\frac{\pi}{2}]^E,$ if there exists a hyperbolic circle-packing metric of $\mathcal{T}$ with discrete conformal factors $u_0$, then for any given $\hat{K}$ satisfying $K(u_0)\le\hat{K}< 2\pi$ pointwise, there exists a discrete conformal factor $u^*$ such that $K(u^*)=\hat{K}$.
\end{thm}
\begin{proof}
    By Theorem \ref{prescribe_conv}, the prescribed curvature flow \eqref{prescribed} will converge with the initial data $u_0.$ Hence, we find the circle-packing metric with prescribed discrete Gaussian curvatures $\hat{K}$.
\end{proof}

As a direct application, we have the following corollary.
\begin{cor}\label{moduli}
    Let $\mathcal{T}$ be a triangulation with bounded degree. Assume that $\Phi\equiv0.$ Then for each $\hat{K}$ satisfying,  for some constant $\epsilon>0$, 
    \[
        (2-\frac{\deg(i)}{3})\pi+\epsilon<\hat{K_i}<2\pi,\quad \forall i\in V,
    \]
  there exists a discrete conformal factor $u^*$ such that $K(u^*)=\hat{K}$.
\end{cor}
\begin{rem}
\begin{enumerate}
    \item In Corollary \ref{moduli}, we assume $\Phi\equiv0$ only for the sake of simplicity in our discussion. When $\Phi$ does not equal $0,$ one could also obtain some existence results. We refer the reader to \cite{MR4761889} for some similar discussion for finite circle packings with a general intersection angle $\Phi$.
    \item An interesting result on the existence of hyperbolic circle-packing metric with finite punctures and prescribed cone angles on noncompact surfaces was proved by Bowers and Ruffoni \cite{bowers2025infinite}.
\end{enumerate}
    
\end{rem}
\section{Convergence of the Euclidean CRF on $\mathcal{T}_H$}
\label{sec:5}

The hexagonal lattice $\mathcal{T}_H=(H,E_H,F_H)$ consists of the set of vertices
\[
H=\{v_{m,n}\in\mathbb{R}^2:v_{m,n}=m+ne^{\frac{2\pi}{3}i},~m,n\in\mathbb{Z}\},
\] the set of edges $E_H=\{\{v_{m,n},v_{m',n'}\}:|v_{m,n}-v_{m',n'}|=1\}, $ and the set of faces which are induced by the embedding of $H$ and $E_H$ into $\R^2.$
%For simplicity, we always write $H=V.$

Let us first introduce some basic definitions on the triangulation $\mathcal{T}_H.$ For $p\in[1,\infty]$, we denote by $l^p(H)$  the $l^p-$space of functions on $H$ with respect to the counting measure, which is a Banach space with the norm

    \begin{eqnarray*}
   \|g\|_{l^p}=\left\{
   \begin{aligned}
   &\Big(\sum_{i\in H}|g_i|^p\Big)^{\frac{1}{p}},~p\in[1,+\infty)\\
   &\sup_{i\in H}|g_i|,~p=\infty.
   \end{aligned}
   \right.
   \end{eqnarray*}

For $l^p$ spaces, one easily obtains the following estimate.
\begin{prop}\label{lp}
    For $1\le p\le q\le+\infty,$ 
    $\|u\|_{l^q}\le\|u\|_{l^p}$.
\end{prop}
The Dirichlet energy of a function $f\in \R^H$ is defined as
$$\mathcal{E}(f):=\sum_{\{i,j\}\in E_H}|f(i)-f(j)|^2.$$
The following proposition for $l^2$ functions on graphs is useful for proving the convergence theorem.
\begin{prop}\label{infinity_converge}
    Let $\{u_n\}_{n=1}^\infty$ be a sequence of $l^2-$functions satisfying 
    $$\|u_n\|_{l^2}\le M,~~\mathcal{E}(u_{n})\rightarrow 0,\quad n\to \infty.$$
    Then 
    $$
    \|u_n\|_{l^\infty}\rightarrow 0,\quad n\to\infty.
    $$
\end{prop}
\begin{proof}
    We prove the proposition by contradiction. If it is not true, then there exists a subsequence $u_{n_k}$ and corresponding vertices $v_{n_k}\in H$ such that 
    \[
    |u_{n_k}(v_{n_k})|>\epsilon,
    \]
    for some constant $\epsilon>0.$ Since $\mathcal{E}(u_{n})\rightarrow 0,$ $$\sup_{\{i,j\}\in E_H}|u_n(i)-u_n(j)|\to 0,\quad n\to \infty.$$ In particular, for each $N>0$, there exists a sufficiently large $K=K(N)$ such that 
    $$
   \sup_{\{i,j\}\in E_H}|u_{n_k}(i)-u_{n_k}(j)|\leq \frac{\epsilon}{2N},~\forall k\ge K.
    $$
    Therefore, $$\|u_{n_k}\|_{l^2}\ge\|u_{n_k}\|_{l^2(B_N(v_{n_k}))} \ge\frac{N\epsilon}{2}.$$
    Since $N$ can be arbitrarily large, this leads to a contradiction.
\end{proof}

\subsection{Discrete Laplacian and the analysis on the hexagonal lattice}
For a function $u\in \R^H,$ we write $u_{m,n}=u(v_{m,n})$ for $m,n\in \Z$ for simplicity. We denote by $\mathcal{D}=\{D_i\}_{i=1}^6$ the set of \textbf{basic difference operators} on $\mathbb{R}^H$ as
\begin{eqnarray*}
   \left\{
   \begin{aligned}
   &        D_1 u_{m,n}=u_{m+1,n}-u_{m,n} , ~D_4 u_{m,n}=-D_1u_{m-1,n},\label{short_range}\\
   &
   D_2 u_{m,n}=u_{m+1,n+1}-u_{m,n} ,~ D_5 u_{m,n}=-D_2u_{m-1,n-1},\\& D_3 u_{m,n}=u_{m,n+1}-u_{m,n} ,~D_6 u_{m,n}=-D_3u_{m,n-1}.
   \end{aligned}
   \right.
   \end{eqnarray*}
   We write $Du:H\to \R^6$ as $$D u_{m,n}=(D_1u_{m,n},D_2u_{m,n},D_3u_{m,n},D_4u_{m,n},D_5u_{m,n},D_6u_{m,n}).$$
We denote by $D^2u:H\to M_{6\times 6}$ the Hessian of a function, a matrix-valued map, given by $$D^2u_{m,n}:=(D_iD_ju_{m,n})_{i,j=1}^6.$$
   
  %We write $D u$ and $D^2u$ for $(D u_{m,n})_{m,n\in\mathbb{Z}}$ and $(D^2 u_{m,n})_{m,n\in\mathbb{Z}}$ respectively.
  For any $p\in [1,\infty),$ we define 
  $$\|Du\|_{l^p}^p:=\frac12\sum_{m,n\in \Z}\sum_{D_i\in \mathcal{D}}|D_i u_{m,n}|^p.$$
  Then one verifies that  $\mathcal{E}(u)=\|D u\|_{l^2}.$
   We set
   $$\|D^2u\|_{l^2}^2:=\sum_{m,n\in\mathbb{Z}}\sum_{D_i,D_j\in \mathcal{D}}(D_iD_ju_{m,n})^2.$$
   One easily checks that
   \begin{align}\label{energyform2}
       \|D^2u\|^2_{l^2}=2\sum_{D_i\in\mathcal{D}}\|D(D_iu)\|_{l^2}^2.
   \end{align}
Since every vertex in $H$ incident to exactly $6$ edges in $E_H,$ there exists a universal constant $C_1$ such that 
\begin{align}\label{sob}
\|D g\|_{l^2}^2\le C_1\|g\|_{l^2}^2,\|D^2 g\|_{l^2}^2\le C_1\|g\|_{l^2}^2,~~\forall g\in l^2(H).
\end{align}

On the hexagonal lattice, we consider the discrete Laplacian of constant weight, i.e. there exists $w_0>0$ such that $w_{ij}=w_0$ for any $\{i,j\}\in E_H$. By \eqref{eq:lap}, \begin{align}\label{laplace}\Delta_{\omega}=\omega_0\sum_{i=1}^6 D_i.\end{align} Note that two basic difference operators on $\mathcal{T}_H$ commute, i.e. for any $D_i,D_j\in \mathcal{D},$ 
\begin{align*}
D_iD_ju=D_iD_ju,~~\forall u\in\mathbb{R}^H.
\end{align*}
%Also, since
%\begin{align}\label{laplace}\Delta_{\omega}=\omega_0(D_1+D_2+D_3+D_4+D_5+D_6),\end{align} 
 Hence, the discrete Laplacian of constant weight commutes with any basic difference operator.
Assuming that $\omega\equiv\omega_0$ is a constant weight on $\mathcal{T}_H.$
For any $f,g\in l^2(H)$, the inner product of $f$ and $g$ given by
\[
(f,g)=\sum_{m,n\in\mathbb{Z}}f_{m,n}g_{m,n}.
\]
We have the following Green formula for $l^2$ functions on the hexagonal lattice,

 \begin{align*}
     (f,\Delta_{\omega} g)=-\omega_0\sum_{\{i,j\}\in E_H}(f(i)-f(j))(g(i)-g(j)),\forall f,g\in l^2(H).
 \end{align*}
 In particular, for $f\in l^2(H),$
 \begin{align*}
   (f,\Delta_{\omega} f)=-\omega_0\|D f\|_{l^2}^2.  
 \end{align*}

\subsection{The Euclidean CRF as a semilinear parabolic equation}
We focus on the flow \eqref{infinite_flow} in Euclidean background geometry. Let $\{i,j,k\}$ be a face of $\mathcal{T}_H.$ Note that $\theta_i^{jk}$ is a function depending on $u_i,u_j$ and $u_k,$ and $\theta_i^{jk}=G(u_j-u_i,u_k-u_i),$ with
\[
G(x,y)=\arccos\frac{(1+e^x)^2+(1+e^y)^2-(e^x+e^y)^2}{2(1+e^x)(1+e^y)}.
\]
Therefore, $K_i(u)=2\pi-\sum_{\{i,j,k\}\in F_H}G(u_k-u_i,u_j-u_i).$ Since when $u\equiv 0$ all triangles of $\mathcal{T}_H$ are equilateral triangles,  $G(0,0)=\frac{\pi}{3}.$
Thus, the flow \eqref{infinite_flow} can be written as, for any $i\in H$ and $t\geq 0,$ 
\begin{align}
    \ddt{u_i}&=-2\pi+\sum_{\{i,j,k\}\in F_H}G(u_k-u_i,u_j-u_i)\nonumber
    \\&=\sum_{\{i,j,k\}\in F_H}[G_x(0,0)(u_j-u_i)+G_y(0,0)(u_k-u_i)+\tilde{F}(u_j-u_i,u_k-u_i)],\nonumber
    \\&=\Delta_{\omega_*} u_i+F(D u)(i)\label{form}
\end{align}
where $\tilde{F}(x,y)=G(x,y)-\frac{\pi}{3}-G_x(0,0)x-G_y(0,0)y,$ and \begin{equation}\label{eq:defF}
    F(D u)(i)=\sum_{\{i,j,k\}\in F_H}\tilde{F}(u_j-u_i,u_k-u_i).
\end{equation}
In fact, one defines $F:\R^6\to\R$ as
$F(z)=\sum_{i=1}^6 \tilde{F}(z_i,z_{i+1})$ with $z=(z_1,\cdots,z_6)$ and $z_7:=z_1.$
Then $F(Du)$ is the composition of $F$ and $Du.$
Since $G$ is a smooth function of $x$ and $y,$ $F(0)=0,~\nabla {F}(0)=0.$ Therefore, we have the following proposition.
\begin{prop}\label{prop:Fest}
    There exists universal constants $\epsilon_1,C_1>0$ such that for any $|z|\leq \epsilon_1,$
    \begin{equation}\label{eq:Fest}
        |F(z)|\leq C_1|z|^2,\quad |\nabla F(z)|\leq C_1|z|.
    \end{equation}
\end{prop}
This yields that $F(D u)(i)=O(|D u(i)|^2),$ when $|D u|$ is small enough. 
Besides, since $G(x,y)$ is symmetric with respect to $x$ and $y$, one has $G_x(0,0)=G_y(0,0).$ Therefore, in \eqref{form}, the weight function $\omega_*$ is given by 
\[
(\omega_*)_{ij}=\omega_0= 2G_x(0,0)=\frac{\sqrt{3}}{3},
\]
which is a constant weight on $\mathcal{T}_H$.

By the discussion above, the CRF  on the hexagonal triangulation in Euclidean background geometry is reduced to the following equation \eqref{semi} with initial value $\phi,$ which is analogous to a semilinear evolution equation. We will prove Theorem \ref{mainthm} by studying asymptotic behaviors of the solutions to such an equation.
\begin{thm}\label{cauchy_problem}
There is a universal constant $\epsilon_0>0$  such that for any $\|\phi\|_{l^2}\leq \epsilon_0,$  the Cauchy problem of the semilinear parabolic equation 
\begin{align}\label{semi}
\begin{cases}
  \ddt{u_i}-\Delta_{\omega_*} u_i=F(D u)(i),~\forall i\in V,\\
    u(0)=\phi,
\end{cases}
\end{align} where $F(Du)$ is defined in \eqref{eq:defF}, 
    admits a unique global solution $u\in L^\infty([0,\infty);l^2(H))$ such that 
\begin{align}\label{decay}
       \EE(u(t))\rightarrow 0,\quad t\to\infty.
    \end{align}
\end{thm}
\begin{rem}

In the continuous case, the existence and uniqueness result similar to Theorem \ref{mainthm} for Cauchy problem of nonlinear evolution equation in $\mathbb{R}^n$ are studied; see \cite[Theorem 5.2.2]{zheng2004nonlinear}.

\end{rem}

For proving Theorem \ref{cauchy_problem}, we first establish energy estimates for the equation \eqref{semi}.
\begin{lem}\label{l2}
    There exists a constant $\epsilon_2$ such that for any $T>0$ if $u(t)$ is a solution to the equation \eqref{cauchy_problem} with
    \begin{align*}
        u(t)\in C^1([0,T];l^2(H)),~\sup_{t\in[0,T]}\|u(t)\|_{l^2}\le \epsilon_2,
    \end{align*}
    then 
    $$\ddt{}\|u(t)\|_{l^2}^2+w_0\EE(u(t))\le 0.$$ Moreover, $\|u(t)\|_{l^2}$ is non-increasing and $\int_0^\infty\EE(u(t))dt<\infty.$
\end{lem}
\begin{proof}
Since $u(t)\in C([0,T];l^2(H)),$ multiplying $u$ on both sides of \eqref{semi} and summing over vertices, we have
\begin{align}
&\frac{1}{2}\ddt{\|u\|_{l^2}^2}+\omega_0\EE(u)=(F(Du),u)\label{homo2}
\end{align}
Choose $\epsilon_2\leq \frac{\epsilon_1}{2},$ where $\epsilon_1$ is given in Proposition \ref{prop:Fest}.

Since  $\|u\|_{l^\infty}\leq \|u\|_{l^2}\le\epsilon_2$, we have for some constant $C_2,$
    \begin{align}\label{small_energy}
    \|F(Du)\|_{l^2}\le C_2\|Du\|_{l^4}^2\le C_2\EE(u).
    \end{align}
    Let $\epsilon_2:=\min\{\frac{\omega_0}{2C_2},\frac{\epsilon_1}{2}\}.$
    Therefore, for $t\geq 0,$ 
    \begin{align*}
        \frac{1}{2}\ddt{\|u\|_{l^2}^2}+\omega_0\EE(u)&\le\|F(Du)\|_{l^2}\|u\|_{l^2}\nonumber\\&\le\frac{\omega_0}{2}\EE(u).
    \end{align*}
    This proves the energy estimate. Then it is obvious that 
    $\|u(t)\|_{l^2}$ is non-increasing and $$\int_0^\infty\EE(u(t))dt\leq \|u(0)\|_{l^2}<\infty.$$ This proves the lemma.
\end{proof}
In addition, we have the following lemma.
\begin{lem}\label{energy_etimate}
    There exists a constant $\epsilon_2$ such that if $u(t)$ is a solution to the equation \eqref{cauchy_problem} with
    \begin{align*}
        u(t)\in C^1([0,\infty);l^2(H)),~\sup_{t\in[0,\infty)}\|u(t)\|_{l^2}\le \epsilon_2,
    \end{align*}
  Then $\ddt{\EE(u)}$ exists and is uniformly bounded in $[0,\infty)$.
    
\end{lem}
\begin{proof}
Let $\epsilon_2$ be given in Lemma~\ref{l2}. 
%Let $\|\phi\|_{l^2}\le\epsilon_0$, by \eqref{int_energy} we see that the inequality \eqref{inte} holds.
    By the equation \eqref{semi}, we have
    \begin{align}\label{homo1}
        \ddt{D_iu}-\Delta_{\omega_*}(D_i u)=D_i F(Du),~\forall D_i\in \mathcal{D}.
    \end{align}
Note that $$\sup_{t\in[0,\infty)}\|Du(t)\|_{l^2}\leq C\sup_{t\in[0,\infty)}\|u(t)\|_{l^2}<\infty.$$  The same energy estimate as in Lemma~\ref{l2} yields that
    \begin{align}
        \frac{1}{2}\ddt{}\|{D_iu}\|_{l^2}^2+\omega_0\EE(D_i u)=(D_iF(Du),D_iu)\label{homo3},~\forall D_i\in \mathcal{D}.
    \end{align}
   Summing over $D_i\in\mathcal{D}$,  we get
    \begin{align}\label{homo4}
    &\left\rvert \ddt{\EE(u)}+\frac{\omega_0}{2}\|D^2 u\|_{l^2}^2\right\rvert=\left\rvert2\frac{1}{\omega_0}(F(Du),\Delta_{\omega_*} u)\right\rvert.\nonumber\\&\le
    C\EE(u)\|D^2u\|_{l^2}\le C'\|u\|_{l^2}^{3},
    \end{align} where we have used \eqref{sob}.
   Therefore, by \eqref{sob},
   $$
   \left|\ddt{\EE(u(t))}\right|\leq \frac{w_0}{2}\|D^2 u\|_{l^2}^2+C'\|u\|_{l^2}^{3}\leq C,\quad \forall t\geq 0.
   $$
   This proves the result.
\end{proof}
We need a classical lemma in the calculus.
\begin{lem}[Barbalat's Lemma]\label{Barbalat's Lemma}
    Let $f\in C([0,\infty);\mathbb{R})$ be a uniformly continuous function. Suppose that $\int_0^\infty f(t)\mathrm{d}t<\infty$, then
    \[
    \lim_{t\rightarrow\infty}f(t)=0.   \]
\end{lem}

Now we prove Theorem~\ref{cauchy_problem}.
\begin{proof}[Proof of Theorem~\ref{cauchy_problem}]
    Let $\epsilon_0$ be a small constant to be chosen. We denote by $P_t$ the semigroup of the Laplacian operator $\Delta_{\omega^*}$. For any $u_0\in l^2(H)$ with $\|u_0\|_{l^2}\le\epsilon_0,$ we consider the space
    \[
    X=\{u\in C([0,1];l^2(H)):u(0)=u_0,\|u\|_{L^\infty([0,1];l^2(H))}\le2\epsilon_0\},
    \] which is a subset of Banach space $C([0,1];l^2(H))$ with the norm $L^\infty([0,1];l^2(H)).$ We will prove that there is a solution, a weak solution in time, of \eqref{semi} in $X.$ Then for small $\epsilon_0,$ by the regularity theory we get $u\in C^1([0,1];l^2(H))$ using the structure of \eqref{semi}.
  For any $u\in X,$ define the mapping $S$ as
    \[
    Su(t)=P_tu(0)+\int_0^tP_{t-s}F(Du(s))\mathrm{d}t,\quad t\in [0,1].
    \]
  Choose $\epsilon_0\leq \frac{\epsilon_1}{4},$ where $\epsilon_1$ is given in Proposition \ref{prop:Fest}. Note that $\|P_t f\|_2\leq \|f\|_2$ for any $f\in l^2(H), t\geq 0.$ 
  Then by \eqref{small_energy} we have
    \[
    \|Su(t)\|_{l^2}\le \|u(0)\|_{l^2}+C_2\EE(u(t))\le\epsilon_0+C_3\epsilon_0^2\leq 2\epsilon_0,~\forall t\in[0,1],
    \] where we choose $\epsilon_0\leq \frac{1}{C_3}.$
    Therefore we have a mapping $S:X\to X.$

Now we need to show that $S$ is a contraction mapping when $\epsilon_0$ is small. Consider two functions $h_1,h_2 \in X.$ For $t\in[0,1],$ by \eqref{eq:Fest} we have
\begin{align}
    \|Sh_1(t)-Sh_2(t)\|_{l^2}&\le\left\|\int_0^tP_{t-s}(F(Dh_1(s))-F(Dh_2(s)))\mathrm{d}s\right\|_{l^2}\nonumber
    \\&\le\int_0^1\left\|\int_0^1\nabla F(\sigma Dh_1(s)+(1-\sigma) Dh_2(s)))\mathrm{d}\sigma\cdot(Dh_1(s)-Dh_2(s))\right\|_{l^2}\mathrm{d}s\nonumber
    \\&\le C_4\epsilon_0\int_0^1\|Dh_1(s)-Dh_2(s)\|_{l^2}\mathrm{d}s\nonumber\\
    &\le C_5\epsilon_0\|h_1-h_2\|_{L^{\infty}([0,1];l^2(H))}\\
    &\le \frac12\|h_1-h_2\|_{L^{\infty}([0,1];l^2(H))},
\end{align}
if $\epsilon_0\leq \frac{1}{2C_5}.$  Hence we choose $\epsilon_0:=\min\{\frac{\epsilon_1}{4},\frac{1}{C_3},\frac{1}{2C_5},\frac{\epsilon_2}{2}\},$ where $\epsilon_2$ is given in Lemma~\ref{l2} and Lemma~\ref{energy_etimate}. This yields that $S$ is a contraction mapping, and by the contraction mapping theorem there is a fixed point, i.e. a solution of \eqref{semi}, $u\in X.$ That is, there is a solution for any initial data $u_0$ with $\|u_0\|_{l^2}\le\epsilon_0.$

For such a solution $u\in X,$ $\|u\|_{L^{\infty}([0,1];l^2(H))}\leq 2\epsilon_0\leq \epsilon_2.$ Hence by Lemma~\ref{l2}, $\|u(1)\|_{l^2}\le \|u(0)\|_{l^2}\leq \epsilon_0.$ Therefore, we can extend the solution to the time interval $[1,2],$ and to all $[n,n+1]$ by induction. We obtain an entire solution 
\[
u(t)\in L^{\infty}([0,\infty),l^2(H)).
\]
By Lemma \ref{energy_etimate} and Lemma  \ref{Barbalat's Lemma}, we have $$\EE(u(t))\rightarrow 0,\quad t\to \infty.$$ This proves the theorem.
\end{proof}

\begin{rem}
    Note that in continuous cases, similar results always require that the $H^s$ norm of the initial data is sufficiently small for some $s>4;$ see e.g. \cite{zheng2004nonlinear}. However, for the discrete case, we only require the initial data to have a small $l^2$ norm.
\end{rem}

Finally, we prove Theorem \ref{mainthm}.
\begin{proof}[\textbf{Proof of Theorem \ref{mainthm}}]

    By Theorem \ref{cauchy_problem}, we obtain that $\|\EE(u(t))\|_{l^2}\rightarrow 0$ as $t\rightarrow \infty.$ Since $\|u(t)\|_{l^2}$ is bounded, by Proposition \ref{infinity_converge} we have $\|u(t)\|_{l^\infty}\rightarrow 0$ as $t\to\infty.$ Moreover, by Theorem \ref{thm:bdd}, the solution to the CRF is unique. This proves the theorem.
\end{proof}

\section{The conjectures}
For Conjecture \ref{conj}, (B) and (C) are equivalent due to the theorem of He. Moreover, one easily sees that (A) implies (B), since the convergence of the hyperbolic CRF guarantees a hyperbolic circle-packing metric. Therefore, one only need to verify that $(C)$ implies $(A)$.

Besides circle packings discussed in our article, ideal circle patterns, related to ideal polyhedra in hyperbolic 3-space $\mathbb{H}^3,$ are of great importance; see \cite{MR1370757,MR2022715}. 
In brief, for an ideal polyhedron $\mathcal{P}$ in $\mathbb{H}^3$, each face $\mathcal{F}$ of $\mathcal{P}$ corresponds to a disc $\mathcal{D_{\mathcal{F}}}$ on the unit sphere $\mathbb{S}^2$ of the Poincar\'e ball model. Then $\{\partial\mathcal{D}_\mathcal{F}\}_{\mathcal{F}~\text{is a face of}~ \mathcal{P}}$ forms an ideal circle pattern $\mathcal{I}$ on $\mathbb{S}^2$. We define the carrier of the ideal polyhedron $\mathcal{P}$ by
\[\mathrm{Carrier}(\mathcal{P})=\cup_{\mathcal{F}~\text{is a face of}~ \mathcal{P}}\mathcal{D}_\mathcal{F}.\]

For ideal circle packings on compact surfaces, combinatorial Ricci flows were introduced before; see \cite[Equation (6),(7)]{ge2021combinatorial}. It seems feasible to define combinatorial Ricci flows for ideal circle patterns on infinite cellular decompositions. We have the following conjecture.
\begin{conjecture}
    Let $\mathcal{T}=(V,E,F)$ be an infinite cellular decomposition of an open disk with intersection angles $\Phi\in (0,\pi)^E$, then the following statements are equivalent:
    \begin{enumerate}[(A')]
        \item The CRF on $\mathcal{T}$ converges for any initial value in hyperbolic background geometry.
        \item There exists a locally finite ideal circle pattern with contact graph $\mathcal{T}$ in the unit disc.
        \item The $1-$skeleton of cellular decomposition $\mathcal{T}$ is VEL-hyperbolic.
        \item There exists an infinite ideal polyhedron $\mathcal{P}$ in $\mathbb{H}^3$ that is combinatorial equivalent to the Poincar\'e dual of $\mathcal{T}$, whose carrier is an open hemisphere of $~\mathbb{S}^2$.
    \end{enumerate}
\end{conjecture}
\textbf{Acknowledgements.} The authors would like to thank Lang Qin for helpful discussions and suggestions for the Ricci flow. H. Ge is supported by NSFC, no.12341102, no.12122119. B. Hua is supported by NSFC, no.12371056, and by Shanghai Science and Technology Program [Project No. 22JC1400100]. 
\section{Appendix}\label{appendix}
We will prove Lemma \ref{est_deri} in this appendix. For triangles in Euclidean background geometry, the lemma was proved by He, see \cite[Lemma 3.2]{HE}. Therefore, it is sufficient to prove Lemma \ref{est_deri} for triangles in hyperbolic background geometry.
\begin{figure}[htbp]
\centering
\includegraphics[scale=0.16]{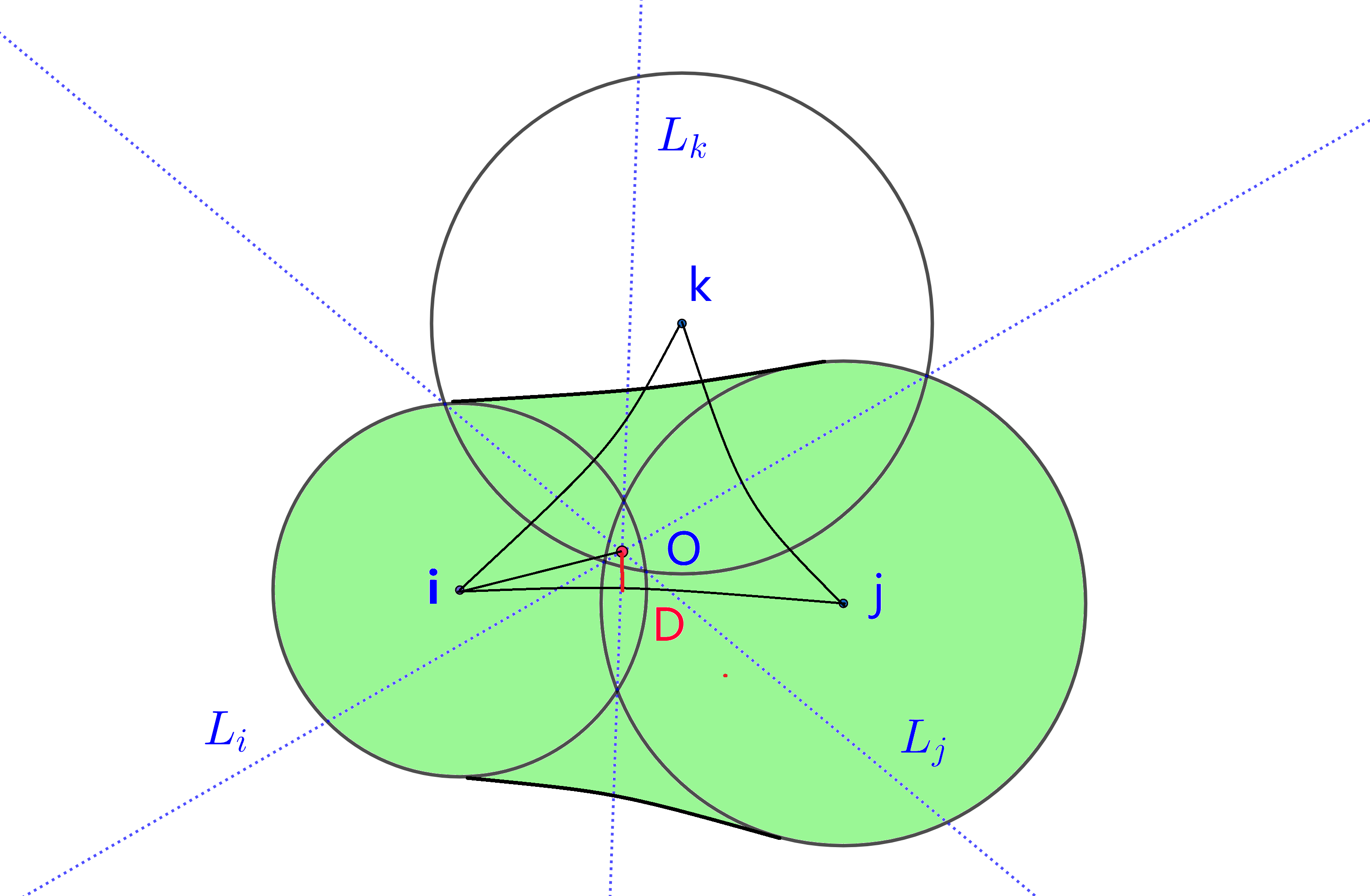}
%\caption{B}
\captionof{figure}{\small A circle packing of a triangle with power center locating at the origin in the Poincar\'e disk model.}
  \label{estimategraph}
\end{figure} 

\begin{proof}[\textbf{Proof of Lemma \ref{est_deri}}]
    Here we first recall the geometric interpretation of the partial derivative $\pp{\theta_i}{u_j}$ for circle packing in hyperbolic background geometry, which can be found in the work of Chow and Luo \cite{2003Combinatorial}.

    Let $\{i,j,k\}$ be a triangle with circle-packing metric and $\mathcal{C}_i, \mathcal{C}_j$ and $\mathcal{C}_k$ be three corresponding circles as shown in Figure \ref{estimategraph}.\ We denote by $L_i,L_j,L_k$ the geodesic lines passing through the pairs of intersection points of $\{\mathcal{C}_j,\mathcal{C}_k\}$, $\{\mathcal{C}_i,\mathcal{C}_k\}$ and $\{\mathcal{C}_i,\mathcal{C}_j\}$ respectively. Then it is known that $L_i, L_j$ and $L_k$ have a common point $O$ in the hyperbolic plane. The point $O$ is called the power center in the literature; see e.g.  \cite{gu2008computational}. If the intersection angles $\Phi\in\mathbb[0,\frac{\pi}{2}]^3$, $O$ is located inside the triangle $\{i,j,k\}.$ One easily sees that $L_k$ is perpendicular to the edge $\{i,j\},$ and let $D$ be the foot of perpendicular as shown in Figure \ref{estimategraph}.   
    Let $l_{OD}$ be the distance from the the power center $O$ to the edge $\{i,j\}.$ It is proved in \cite{2003Combinatorial} that 
    \[
    \pp{\theta_i^{jk}}{u_j}=\frac{\tanh{l_{OD}}}{\sinh{l_{ij}}}.
    \]
Therefore, by the trigonometry of hyperbolic triangles $\tanh{l_{OD}}=\tan\angle Oij\sinh l_{iD}$, we have
\[
\pp{\theta_i^{jk}}{u_j}\le\tan\angle Oij\le\tan \theta_i^{jk}.
\]
Therefore, there exists a $C_1>0$ such that 
\begin{align}\label{estimate_length}
    \pp{\theta_i^{jk}}{u_j}\le C_1\theta_i^{jk} ~~\text{if}~~\theta_i^{jk}\in[0,1].
\end{align}
Since $O$ is contained in the geodesic convex hull of circles $\mathcal{C}_i$ and $\mathcal{C}_j$, which is colored in green in Figure \ref{estimategraph}, one easily sees that $l_{OD}\le \max\{r_i,r_j\}$. Therefore, we have
\[
\pp{\theta_i^{jk}}{u_j}\le\frac{\sinh{l_{OD}}}{\sinh{l_{ij}}} \le1.
\]
Hence, by \eqref{estimate_length} we conclude that $\pp{\theta_i^{jk}}{u_j}\le C\theta_i^{jk}$ for some constant $C>0.$

\end{proof}

\bibliographystyle{plain}
\bibliography{reference}

\begin{thebibliography}{10}

\bibitem{MR3958792}
Richard~H. Bamler, Esther Cabezas-Rivas, and Burkhard Wilking.
\newblock The {R}icci flow under almost non-negative curvature conditions.
\newblock {\em Invent. Math.}, 217(1):95--126, 2019.

\bibitem{MR2022715}
Alexander~I. Bobenko and Boris~A. Springborn.
\newblock Variational principles for circle patterns and {K}oebe's theorem.
\newblock {\em Trans. Amer. Math. Soc.}, 356(2):659--689, 2004.

\bibitem{bowers2025infinite}
P.~L. Bowers and L.~Ruffoni.
\newblock Infinite circle packings on surfaces with conical singularities.
\newblock {\em Comput. Geom.}, 127:Paper No. 102160, 13, 2025.

\bibitem{MR3429162}
Esther Cabezas-Rivas and Burkhard Wilking.
\newblock How to produce a {R}icci flow via {C}heeger-{G}romoll exhaustion.
\newblock {\em J. Eur. Math. Soc. (JEMS)}, 17(12):3153--3194, 2015.

\bibitem{MR3480020}
Albert Chau, Ka-Fai Li, and Luen-Fai Tam.
\newblock Deforming complete {H}ermitian metrics with unbounded curvature.
\newblock {\em Asian J. Math.}, 20(2):267--292, 2016.

\bibitem{MR2520796}
Bing-Long Chen.
\newblock Strong uniqueness of the {R}icci flow.
\newblock {\em J. Differential Geom.}, 82(2):363--382, 2009.

\bibitem{MR2260930}
Bing-Long Chen and Xi-Ping Zhu.
\newblock Uniqueness of the {R}icci flow on complete noncompact manifolds.
\newblock {\em J. Differential Geom.}, 74(1):119--154, 2006.

\bibitem{Chow1}
Bennett Chow.
\newblock The {R}icci flow on the {$2$}-sphere.
\newblock {\em J. Differential Geom.}, 33(2):325--334, 1991.

\bibitem{2003Combinatorial}
Bennett Chow and Feng Luo.
\newblock Combinatorial {R}icci flows on surfaces.
\newblock {\em J. Differential Geom.}, 63(1):97--129, 2003.

\bibitem{chung1997spectral}
Fan R.~K. Chung.
\newblock {\em Spectral graph theory}, volume~92 of {\em CBMS Regional Conference Series in Mathematics}.
\newblock Conference Board of the Mathematical Sciences, Washington, DC; by the American Mathematical Society, Providence, RI, 1997.

\bibitem{MR1106755}
Yves Colin~de Verdi\`ere.
\newblock Un principe variationnel pour les empilements de cercles.
\newblock {\em Invent. Math.}, 104(3):655--669, 1991.

\bibitem{MR1371208}
Panagiota Daskalopoulos and Manuel~A. del Pino.
\newblock On a singular diffusion equation.
\newblock {\em Comm. Anal. Geom.}, 3(3-4):523--542, 1995.

\bibitem{MR4466650}
Ke~Feng, Huabin Ge, and Bobo Hua.
\newblock Combinatorial {R}icci flows and the hyperbolization of a class of compact 3-manifolds.
\newblock {\em Geom. Topol.}, 26(3):1349--1384, 2022.

\bibitem{ge_phd}
H.~Ge.
\newblock Combinatorial methods and geometric equations.
\newblock {\em Ph.D. Thesis, Peking University, Beijing}, 2012.

\bibitem{MR4334399}
Huabin Ge, Bobo Hua, and Ze~Zhou.
\newblock Circle patterns on surfaces of finite topological type.
\newblock {\em Amer. J. Math.}, 143(5):1397--1430, 2021.

\bibitem{ge2021combinatorial}
Huabin Ge, Bobo Hua, and Ze~Zhou.
\newblock Combinatorial {R}icci flows for ideal circle patterns.
\newblock {\em Adv. Math.}, 383:Paper No. 107698, 26, 2021.

\bibitem{MR4024520}
Huabin Ge and Wenshuai Jiang.
\newblock On the deformation of inversive distance circle packings, {I}.
\newblock {\em Trans. Amer. Math. Soc.}, 372(9):6231--6261, 2019.

\bibitem{MR4761889}
Huabin Ge and Aijin Lin.
\newblock The character of {T}hurston's circle packings.
\newblock {\em Sci. China Math.}, 67(7):1623--1640, 2024.

\bibitem{MR3269185}
Huabin Ge and Xu~Xu.
\newblock Discrete quasi-{E}instein metrics and combinatorial curvature flows in 3-dimension.
\newblock {\em Adv. Math.}, 267:470--497, 2014.

\bibitem{MR2832165}
Gregor Giesen and Peter~M. Topping.
\newblock Existence of {R}icci flows of incomplete surfaces.
\newblock {\em Comm. Partial Differential Equations}, 36(10):1860--1880, 2011.

\bibitem{MR860324}
A.~A. Grigor'yan.
\newblock Stochastically complete manifolds.
\newblock {\em Dokl. Akad. Nauk SSSR}, 290(3):534--537, 1986.

\bibitem{MR3822363}
A.~A. Grigor’yan.
\newblock {\em Introduction to analysis on graphs}, volume~71 of {\em University Lecture Series}.
\newblock American Mathematical Society, Providence, RI, 2018.

\bibitem{MR3825607}
Xianfeng Gu, Ren Guo, Feng Luo, Jian Sun, and Tianqi Wu.
\newblock A discrete uniformization theorem for polyhedral surfaces {II}.
\newblock {\em J. Differential Geom.}, 109(3):431--466, 2018.

\bibitem{MR3807319}
Xianfeng~David Gu, Feng Luo, Jian Sun, and Tianqi Wu.
\newblock A discrete uniformization theorem for polyhedral surfaces.
\newblock {\em J. Differential Geom.}, 109(2):223--256, 2018.

\bibitem{gu2008computational}
Xianfeng~David Gu and Shing-Tung Yau.
\newblock {\em Computational conformal geometry}, volume~3 of {\em Advanced Lectures in Mathematics (ALM)}.
\newblock International Press, Somerville, MA; Higher Education Press, Beijing, 2008.
\newblock With 1 CD-ROM (Windows, Macintosh and Linux).

\bibitem{ham}
Richard~S. Hamilton.
\newblock The {R}icci flow on surfaces.
\newblock In {\em Mathematics and general relativity ({S}anta {C}ruz, {CA}, 1986)}, volume~71 of {\em Contemp. Math.}, pages 237--262. Amer. Math. Soc., Providence, RI, 1988.

\bibitem{MR4494617}
Robert Haslhofer.
\newblock Uniqueness and stability of singular {R}icci flows in higher dimensions.
\newblock {\em Proc. Amer. Math. Soc.}, 150(12):5433--5437, 2022.

\bibitem{HE}
Zheng-Xu He.
\newblock Rigidity of infinite disk patterns.
\newblock {\em Ann. of Math. (2)}, 149(1):1--33, 1999.

\bibitem{He_schramm}
Zheng-Xu He and O.~Schramm.
\newblock Hyperbolic and parabolic packings.
\newblock {\em Discrete Comput. Geom.}, 14(2):123--149, 1995.

\bibitem{Doyle_spirals}
Bobo Hua and Puchun Zhou.
\newblock The rigidity of doyle circle packings on the infinite hexagonal triangulation.
\newblock {\em arXiv preprint arXiv:2404.11258}, 2024.

\bibitem{MR2538937}
James Isenberg and Mohammad Javaheri.
\newblock Convergence of {R}icci flow on {$\Bbb R^2$} to flat space.
\newblock {\em J. Geom. Anal.}, 19(4):809--816, 2009.

\bibitem{MR2545867}
Lizhen Ji, Rafe Mazzeo, and Natasa Sesum.
\newblock Ricci flow on surfaces with cusps.
\newblock {\em Math. Ann.}, 345(4):819--834, 2009.

\bibitem{koebe1936origin}
P.~Koebe.
\newblock {\em Kontaktprobleme der konformen Abbildung}.
\newblock Hirzel Stuttgart, 1936.

\bibitem{MR4015429}
Man-Chun Lee.
\newblock On the uniqueness of {R}icci flow.
\newblock {\em J. Geom. Anal.}, 29(4):3098--3112, 2019.

\bibitem{MR834612}
Peter Li and Shing-Tung Yau.
\newblock On the parabolic kernel of the {S}chr\"odinger operator.
\newblock {\em Acta Math.}, 156(3-4):153--201, 1986.

\bibitem{luo2004combinatorial}
Feng Luo.
\newblock Combinatorial {Y}amabe flow on surfaces.
\newblock {\em Commun. Contemp. Math.}, 6(5):765--780, 2004.

\bibitem{MR3049633}
Li~Ma.
\newblock Convergence of {R}icci flow on {$R^2$} to the plane.
\newblock {\em Differential Geom. Appl.}, 31(3):388--392, 2013.

\bibitem{MR1370757}
Igor Rivin.
\newblock A characterization of ideal polyhedra in hyperbolic {$3$}-space.
\newblock {\em Ann. of Math. (2)}, 143(1):51--70, 1996.

\bibitem{Rodin_Sullivan}
Burt Rodin and Dennis Sullivan.
\newblock The convergence of circle packings to the {R}iemann mapping.
\newblock {\em J. Differential Geom.}, 26(2):349--360, 1987.

\bibitem{MR1244661}
Oded Schramm.
\newblock Square tilings with prescribed combinatorics.
\newblock {\em Israel J. Math.}, 84(1-2):97--118, 1993.

\bibitem{Shi_noncompact}
Wan-Xiong Shi.
\newblock Deforming the metric on complete {R}iemannian manifolds.
\newblock {\em J. Differential Geom.}, 30(1):223--301, 1989.

\bibitem{Stephenson_intro}
Kenneth Stephenson.
\newblock {\em Introduction to circle packing}.
\newblock Cambridge University Press, Cambridge, 2005.
\newblock The theory of discrete analytic functions.

\bibitem{thurston1980geometry}
W.~P. Thurston.
\newblock The geometry and topology of three-manifolds.
\newblock {\em \href{http://www.msri.org/gt3m}{http://www.msri.org/gt3m}}, 1980.

\bibitem{MR3728651}
Peter~M. Topping.
\newblock Ricci flows with unbounded curvature.
\newblock In {\em Proceedings of the {I}nternational {C}ongress of {M}athematicians---{S}eoul 2014. {V}ol. {II}}, pages 1033--1048. Kyung Moon Sa, Seoul, 2014.

\bibitem{MR3352241}
Peter~M. Topping.
\newblock Uniqueness of instantaneously complete {R}icci flows.
\newblock {\em Geom. Topol.}, 19(3):1477--1492, 2015.

\bibitem{ToppingYin24}
Peter~M. Topping and Hao Yin.
\newblock Uniqueness of {R}icci flows from non-atomic {R}adon measures on {R}iemann surfaces.
\newblock {\em Proc. Lond. Math. Soc. (3)}, 128(6):Paper No. e12600, 34, 2024.

\bibitem{wu1993ricci}
Lang-Fang Wu.
\newblock The {R}icci flow on complete {${\bf R}^2$}.
\newblock {\em Comm. Anal. Geom.}, 1(3-4):439--472, 1993.

\bibitem{MR3091259}
Guoyi Xu.
\newblock Short-time existence of the {R}icci flow on noncompact {R}iemannian manifolds.
\newblock {\em Trans. Amer. Math. Soc.}, 365(11):5605--5654, 2013.

\bibitem{MR2520032}
Hao Yin.
\newblock Normalized {R}icci flow on nonparabolic surfaces.
\newblock {\em Ann. Global Anal. Geom.}, 36(1):81--104, 2009.

\bibitem{zheng2004nonlinear}
Songmu Zheng.
\newblock {\em Nonlinear evolution equations}, volume 133 of {\em Chapman \& Hall/CRC Monographs and Surveys in Pure and Applied Mathematics}.
\newblock Chapman \& Hall/CRC, Boca Raton, FL, 2004.

\bibitem{zhou2023generalizing}
Ze~Zhou.
\newblock Generalizing andreev's theorem via circle patterns.
\newblock {\em arXiv preprint arXiv:2308.14386}, 2023.

\bibitem{MR3156988}
Xiaorui Zhu.
\newblock Ricci flow on open surface.
\newblock {\em J. Math. Sci. Univ. Tokyo}, 20(3):435--444, 2013.

\end{thebibliography}

\noindent Huabin Ge, hbge@ruc.edu.cn\\
\emph{School of Mathematics, Renmin University of China, Beijing 100872, P. R. China}\\[-8pt]

\noindent Bobo Hua, bobohua@fudan.edu.cn\\[2pt]
\emph{School of Mathematical Sciences, LMNS, Fudan University, Shanghai, 200433, P.R. China}
\\
\noindent Puchun Zhou, pczhou22@m.fudan.edu.cn
\emph{School of Mathematical Sciences, Fudan University, Shanghai, 200433, P.R. China}
\end{document}